\documentclass[letterpaper, 9 pt, conference]{ieeeconf}      

\IEEEoverridecommandlockouts                              

\usepackage{graphics} 
\usepackage{epsfig} 
\usepackage{times} 
\usepackage{amsmath} 
\usepackage{amssymb}  
\usepackage{subfigure}
\usepackage{url}
\usepackage{dsfont}
\newtheorem{theorem}{Theorem}

\usepackage{cite}

\usepackage{color}
\usepackage{algorithm}
\newtheorem{properties}[theorem]{Properties}

\newtheorem{definition}[theorem]{Definition}

\newtheorem{proposition}[theorem]{Proposition}

{ 

}

\newcommand{\bi}{\begin{itemize}}
\newcommand{\ei}{\end{itemize}}
\newcommand{\bd}{\begin{displaymath}}
\newcommand{\ed}{\end{displaymath}}
\newcommand{\be}{\begin{eqnarray*}}
\newcommand{\ee}{\end{eqnarray*}}

\title{\LARGE \bf
On Robust Computation of Koopman Operator and Prediction in Random Dynamical Systems}
\author{Subhrajit Sinha, Bowen Huang, and  Umesh Vaidya\\
\thanks{Financial support from the National Science Foundation grant  ECCS-1150405 and CNS-1329915 is gratefully acknowledged. S. Sinha and U. Vaidya is with the Department of Electrical \& Computer Engineering,
Iowa State University, Ames, IA 50011
        {\tt\small ugvaidya@iastate.edu}}%
}

\begin{document}
\maketitle

\begin{abstract}

In the paper, we consider the problem of robust approximation of transfer Koopman and Perron-Frobenius (P-F) operators from noisy time series data. In most applications, the time-series data obtained from simulation or experiment is corrupted with either measurement or process noise or both. The existing results show the applicability of algorithms developed for the finite dimensional approximation of deterministic system to a random uncertain case. However, these results hold true only in asymptotic and under the assumption of infinite data set. In practice the data set is finite, and hence it is important to develop algorithms that explicitly account for the presence of uncertainty in data-set. 
We propose a robust optimization-based framework for the robust approximation of the transfer operators, where the uncertainty in data-set is treated as deterministic norm bounded uncertainty. The robust optimization leads to a min-max type optimization problem for the approximation of transfer operators. This robust optimization problem is shown to be equivalent to regularized least square problem. This equivalence between robust optimization problem and regularized least square problem allows us to comment on various interesting properties of the obtained solution using robust optimization. In particular, the robust optimization formulation captures inherent tradeoffs between the quality of approximation and complexity of approximation. These tradeoffs are necessary to balance for the proposed application of transfer operators, for the design of optimal predictor. Simulation results demonstrate that our proposed robust approximation algorithm performs better than the Extended Dynamic Mode Decomposition (EDMD) and DMD algorithms for a system with process and measurement noise.

\end{abstract}

\section{Introduction}

There is increased research trend towards application of transfer operator theoretic methods involving transfer Perron-Frobenius and Koopman operators for the analysis and control of nonlinear systems
\cite{Dellnitz_Junge,Mezic2000,froyland_extracting,Junge_Osinga,Mezic_comparison,Dellnitztransport,mezic2005spectral,Mehta_comparsion_cdc,Vaidya_TAC,raghunathan2014optimal,susuki2011nonlinear,mezic_koopmanism,mezic_koopman_stability,surana_observer}.
The basic idea behind these methods is to shift the focus from the state space where the system evolution is nonlinear to measure space or space of functions where the system evolution is linear. The linearity of the transfer operator framework offers several advantages for analysis and design problems involving nonlinear systems. Furthermore, these methods are amicable to data-driven analysis, where the finite dimensional approximation of Koopman and P-F operators can be constructed from times-series data obtained from simulation or experiment. 
The success of the operator theoretic framework relies on the ability to form an accurate finite dimensional approximation of these operators. Towards this goal various data-driven methods are proposed for the finite dimensional approximation of these operators \cite{dellnitz2002set, Mezic2000,DMD_schmitt,rowley2009spectral,EDMD_williams} with Dynamic Mode Decomposition (DMD) and extended DMD being the popular ones. By exploiting the duality between Koopman and P-F operators the work in \cite{Umesh_NSDMD} provides novel naturally structured DMD algorithm for data-driven approximation of both Koopman and P-F operator that preserves positivity and Markov properties of these operators. 

Recent work has focused on the data-driven approximation of Koopman operator for random dynamical systems (RDS) \cite{mezic_stochastic_koopman_spectrum}, \cite{PhysRevE.96.033310}. In \cite{mezic_stochastic_koopman_spectrum} the authors have provided characterization of the spectrum and eigenfunctions of the Koopman operator for discrete and continuous time RDS, while in \cite{PhysRevE.96.033310}, the authors have provided an algorithm to compute the Koopman operator for systems with both process and observation noise. The results in \cite{PhysRevE.96.033310} claim that DMD algorithm will approximate Koopman operator for RDS  in the asymptotic limit of large data set. Furthermore, \cite{PhysRevE.96.033310} proposed system theoretic-based subspace method for the approximation of Koopman operator with measurement noise where DMD does not perform well. However, in practice the data set is finite, and it is essential to account for the presence of uncertainty explicitly in the algorithm. In this paper, we propose a robust approximation of transfer operators using the robust optimization-based framework. Robust optimization-based approach treats uncertainty in the time-series data as deterministic norm bounded uncertainty resulting in min-max type optimization problem for the finite dimensional approximation of transfer operators. In particular, a robust approximation of transfer operator leads to robust least square optimization problem where the uncertainty acts as an adversary which tries to maximize the least square error. Simulation results suggest that the proposed robust optimization-based approach leads to a better approximation of transfer operators for RDS with both process and measurement noise compared to DMD or subspace DMD method.     

Existing results from  optimization theory establish equivalence between robust least square problem and a least square problem with regularization term \cite{caramanis201214}. This equivalence has an interesting interpretation of the solution obtained using proposed robust optimization. In particular, the regularization term imposes sparsity structure on the approximation. More importantly, the regularization term allows us to achieve a trade-off between the quality of approximation and complexity of approximating function. This feature of the robust approximation solution has significant consequence towards the application of transfer operator for the design of data-driven predictor \cite{korda_mezic_predictor}. In particular, the regularization term prevents over-fitting of training data to model parameters thereby enabling better prediction on the test data set. We exploit this property of robust optimization-based approximation framework to propose a data-driven predictor for a nonlinear system. 

The organization of the paper is as follows. In Section \ref{section_operator}, we provide a brief overview of transfer Perron-Frobenius and Koopman operator for random dynamical systems and discuss properties of these two operators. In Section \ref{section_main}, we present the main results on robust optimization-based framework for robust approximation of transfer Koopman and P-F operators. Results on application of the developed framework for the design of data-driven predictor for nonlinear systems are discussed in Section \ref{section_predictor}.
Simulation results are presented in Section \ref{section_simulation} followed by conclusions in Section \ref{section_conclusion}.


\section{Transfer operators for stochastic system}\label{section_operator}

Consider a discrete time random dynamical system of the form
\begin{eqnarray}
x_{t+1}=T(x_t,\xi_t)\label{system}
\end{eqnarray}
where $T:X\times W \to  X$ with $X\subset \mathbb{R}^N$ is assumed to be invertible with respect to $x$ for each fixed value of $\xi$ and smooth diffeomorphism. $\xi_t\in W$ is assumed to be independent identically distributed (i.i.d) random variable drawn from probability distribution $\vartheta$ i.e., 
\[{\rm Prob}(\xi_t\in B)=\vartheta(B)\]
for every set $B\subset W$ and all $t$.
Furthermore, we denote by ${\cal B}(X)$ the Borel-$\sigma$ algebra on $X$ and ${\cal M}(X)$ the vector space of bounded complex-valued measure on $X$.  Associated with this discrete time dynamical system are two linear operators namely Koopman and Perron-Frobenius (P-F) operator. These two operators are defined as follows.
\begin{definition}[Perron-Frobenius Operator]  $\mathbb{P}:{\cal M}(X)\to {\cal M}(X)$ is given by

\begin{eqnarray}
[\mathbb{P}\mu](A)=\int_{{\cal X} }\int_W\chi_{A}(T(x,v))d\vartheta(v)d\mu(x)=&\int_X p(x,A)d\mu(x)
\end{eqnarray}
where $\chi_A(x)$ is the indicator function for set $A$ and $p(x,A)$ is the transition probability function.  \end{definition}
For deterministic dynamical system $p(x,A)=\delta_{T(x)}(A)$. Under the assumption that $p(x,\cdot)$ is absolutely continuous with respect to Lebesgue measure, $m$, we can write 
\[p(x,A)=\int_A k(x,y)dm(y)\]
for all $A\subset X$. Under this absolutely continuous assumption, the P-F operator on the space of densities $L_1(X)$ can be written as  \footnote{with some abuse of notation we are using the same notation for the P-F operator defined on the space of measure and densities.}
\[[\mathbb{P}g](y)=\int_X k(x,y)g(x)dm(x)\]
\begin{definition}[Invariant measures] Invariant measures are the fixed points of
the P-F operator $\mathbb{P}$ that are additionally probability measures. Let $\bar \mu$ be the invariant measure then, $\bar \mu$ satisfies
\[\mathbb{P}\bar \mu=\bar \mu\]
\end{definition}
Under the assumption that the state space $X$ is compact, it is known that the P-F operator admits at least one invariant measure.
\begin{definition} [Koopman Operator] Given any $h\in\cal{F}$, $\mathbb{U}:{\cal F}\to {\cal F}$ is defined by
\[[\mathbb{U} h](x)={\bf E}_{\xi}[h(T(x,\xi))]=\int_W h(T(x,v))d\vartheta(v)\]
\end{definition}

\begin{properties}\label{property}
Following properties for the Koopman and Perron-Frobenius operators can be stated.

\begin{enumerate}
\item [a).] For any function $h\in{\cal F}$ such that $h\geq 0$, we have $[\mathbb{U}h](x)\geq 0$ and hence Koopman is a positive operator.

\item [d).] If we define P-F operator act on the space of densities i.e., $L_1(X)$ and Koopman operator on space of $L_\infty(X)$ functions, then it can be shown that the P-F and Koopman operators are dual to each others as follows 
\begin{eqnarray*}
&&\left<\mathbb{U} f,g\right>=\left<f,\mathbb{P} g\right>
\end{eqnarray*}
where $f\in L_{\infty}(X)$ and $g\in L_1(X)$.

\item [e).] For $g(x)\geq 0$, $[\mathbb{P}g](x)\geq 0$.

\item [f).] Let $(X,{\cal B},\mu)$ be the measure space where $\mu$ is a positive but not necessarily the invariant measure, then the P-F operator  satisfies  following property.
\[\int_X [\mathbb{P}g](x)d\mu(x)=\int_X g(x)d\mu(x)\]\label{Markov_property}
\end{enumerate}
\end{properties}

\section{Robust Approximation of Koopman and P-F operator}\label{section_main}

In this section, we derive the robust version of Extended Dynamic Mode Decomposition. Results for the robust implementation of DMD, Kernel EDMD and NSDMD will follow exactly along similar lines. 
Consider snapshots of data set obtained from simulating a discrete time random dynamical system $x\to T(x,\xi)$ or from an experiment
\begin{eqnarray}
 X = [x_0,x_2,\ldots,x_M]
 \label{data}
\end{eqnarray}
where $x_i\in X\subset \mathbb{R}^n$. The data-set $\{x_k\}$ can be viewed as sample path trajectory generated by random dynamical system and could be corrupted by either process or measurement noise or both. A large number of sample path trajectories need to be simulated to realize sufficient statistics of the random dynamical system. However, in practice, only few sample path trajectories over finite time horizon are available, and it is hard to approximate the statistics of RDS using the limited amount of data-set. Furthermore, rarely one knows the probability distribution of the underlying noise process, i.e., $\vartheta$. Estimating $\vartheta$ is in itself a challenging problem. In spite of these difficulties, it is essential to develop an algorithm for the approximation of transfer operators that explicitly account for the uncertainty in data-set. We propose a robust optimization-based approach to address this challenge. In particular, we consider deterministic, but norm bounded uncertainty in the data set. Since the trajectory $\{x_k\}$ is one particular realization of the RDS, the other random realization can be assumed to be obtained by perturbing $\{x_k\}$. We assume that the data points $x_k$ are perturbed by norm bounded deterministic perturbation of the form 
\[\delta x_k=x_k+\delta,\;\;\; \delta\in \Delta.\]
Several possible choices for the uncertainty set $\Delta$ can be considered. For example
\begin{eqnarray}
\Delta:=\{\delta \in \mathbb{R}^n:\;\; \parallel \delta\parallel_2\leq \rho\}\label{set1}
\end{eqnarray}
restrict the $2$-norm of $\delta$ to $\rho$. Another possible choice could be 
\begin{eqnarray}
\Delta:=\{\delta \in \mathbb{R}^n:\;\; \parallel \delta\parallel_{Q_i}\leq 1,\;\;i=1,\ldots, d\}\label{set2}
\end{eqnarray}
where $Q_i\geq 0$ and implies that uncertainty $\delta$ lies at the intersection of ellipsoids. More generally, one can also consider $\Delta$ set to be of the form
\begin{eqnarray}\Delta=\{\delta\in \mathbb{R}^n: h_i(\delta)\leq 0,\;\;i=1,\ldots,d\}\label{set3}
\end{eqnarray}
for some convex function $h_i(\delta)$.
 These different choices for the uncertainty set $\Delta$ allow us to encapsulate the information about the uncertainty $\delta$.  For example, if the vector random variable in $\mathbb{R}^2$ is uniformly distributed with support in interval $[-d,d]\times [-d,d]$, then the uncertainty set $\Delta$ can be written as intersection of two ellipsoids $\mathcal{E}_1 = \{(x,y)| \frac{x^2}{a^2}+\frac{y^2}{b^2}\leq 1\}$ and $\mathcal{E}_2 = \{(x,y)| \frac{x^2}{b^2}+\frac{y^2}{a^2}\leq 1\}$ where $a>>d$ and ${1}/{a^2} + {1}/{b^2} = {1}/{d^2}$(Fig. \ref{ellipse_uncertainty}(a)). For example, let $d = 0.5$. Then $[-0.5,0.5]\times [-0.5,0.5]$ can be expressed as the intersection of the two ellipsoids $\mathcal{E}_1$ and $\mathcal{E}_2$, with $a=50$ and $b=0.4996$, as shown in figure \ref{ellipse_uncertainty}(b).
\begin{figure}[htp!]
\centering
\subfigure[]{\includegraphics[scale=.225]{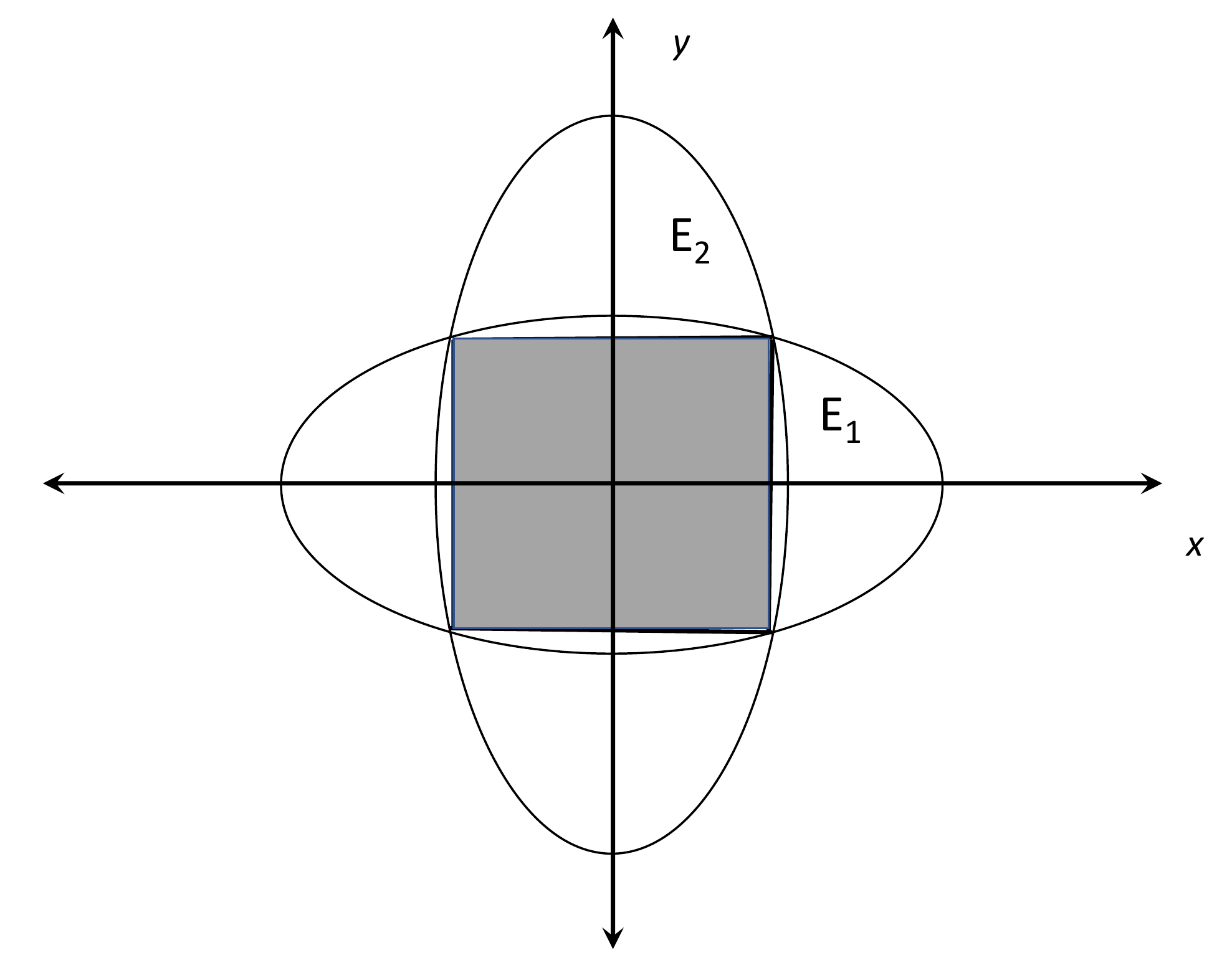}}
\subfigure[]{\includegraphics[scale=.2]{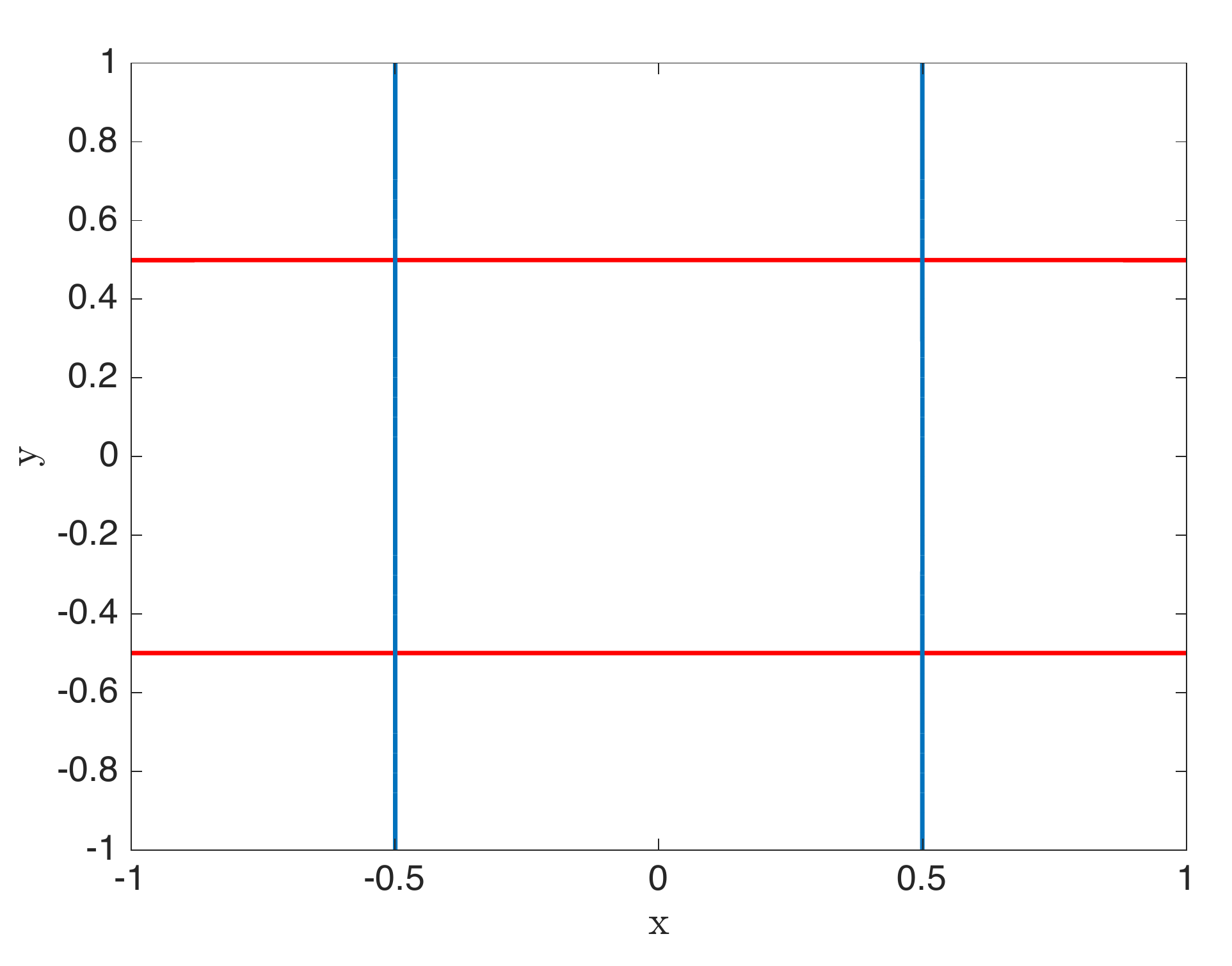}}
\caption{(a) Uncertainty set as intersection of two ellipsoids. (b) Representation of the uncertainty set $\Delta = [-0.5,0.5]\times [-0.5,0.5]$.}\label{ellipse_uncertainty}
\end{figure}


Now let $\mathcal{D}=
\{\psi_1,\psi_2,\ldots,\psi_K\}$ be the set of dictionary functions or observables. The dictionary functions are assumed to belong to $\psi_i\in L_2(X,{\cal B},\mu)={\cal G}$, where $\mu$ is some positive measure, not necessarily the invariant measure of $T$. Let ${\cal G}_{\cal D}$ denote the span of ${\cal D}$ such that ${\cal G}_{\cal D}\subset {\cal G}$. The choice of dictionary functions are very crucial and it should be rich enough to approximate the leading eigenfunctions of Koopman operator. Define vector valued function $\mathbf{\Psi}:X\to \mathbb{C}^{K}$ as
\begin{equation}
\mathbf{\Psi}(x):=\begin{bmatrix}\psi_1(x) & \psi_2(x) & \cdots & \psi_K(x)\end{bmatrix}\label{dic_function}
\end{equation}
In this application, $\mathbf{\Psi}$ is the mapping from physical space to feature space. Any function $\phi,\hat{\phi}\in \mathcal{G}_{\cal D}$ can be written as
\begin{eqnarray}
\phi = \sum_{k=1}^K a_k\psi_k=\boldsymbol{\Psi a},\quad \hat{\phi} = \sum_{k=1}^K \hat{a}_k\psi_k=\boldsymbol{\Psi \hat{a}}\label{expand}
\end{eqnarray}
for some set of coefficients $\boldsymbol{a},\boldsymbol{\hat{a}}\in \mathbb{C}^K$. 
Let \begin{eqnarray}
 \hat{\phi}(x)=[\mathbb{U}\phi](x)+r=E_\xi[\phi(T(x,\xi))]+r. \label{residual}
\end{eqnarray}
Unlike deterministic case where we evaluate (\ref{residual}) at the data point $\{x_k\}$, for the uncertain case  we do not have sufficient data points to evaluate the expected value in the above expression. Instead we use the fact that different realizations of the RDS will consist of the form $\{x_k+\delta\}$ with $\delta\in \Delta$ to write (\ref{residual}) as follows:
\begin{eqnarray}
\hat{\phi}(x_m + \delta x_m)=\phi(x_{m+1})+r, \;\;\;k=1,\ldots,M-1.
\end{eqnarray}
The objective is to minimize the residual for not just one pair of data points $\{x_m,x_{m+1}\}$, but over all possible pairs of data points of the form $\{x_m+\delta,x_{m+1}\}$. 
Using (\ref{expand}) we write the above as follows:
\[
\boldsymbol{\Psi}(x_k + \delta x_k)\boldsymbol {\hat{a}}=\boldsymbol{\Psi}(x_{k+1})\boldsymbol {a}+r.
\]

We seek to find matrix $\bf K$, the finite dimensional approximation of Koopman operator that maps coefficient vector $\boldsymbol{a}$ to $\boldsymbol{\hat{a}}$, i.e., ${\bf K}\boldsymbol{a}=\boldsymbol{\hat a}$, while minimizing the residual term, $r$.  Premultiplying by $\boldsymbol{\Psi}^\top( x_m)$ on both the sides of above expression and summing over $m$ we obtain 
\[\left[\frac{1}{M}\sum_m \boldsymbol{\Psi}^\top( x_m) \boldsymbol{\Psi}(x_m + \delta x_m){\bf K}-\boldsymbol{\Psi}^\top(x_m)\boldsymbol{\Psi}( x_{m+1})\right]{\boldsymbol a}.\]
In the absence of the uncertainty the objective is to minimize the appropriate norm of the  quantity inside the bracket over all possible choices of matrix $\bf K$. However, for robust approximation, presence of uncertainty acts as an adversary whose goal is to maximize the residual term. Hence the robust optimization problem can be formulated as a $\min-\max$ optimization problem as follows. 

\begin{equation}\label{edmd_robust}
\min\limits_{\bf K}\max_{\delta\in \Delta}\parallel {\bf G}_\delta{\bf K}-{\bf A}\parallel_F=:\min\limits_{\bf K}\max_{\delta\in \Delta} {\cal F}({\bf K}, {\bf G}_\delta,{\bf A})
\end{equation}
where
\begin{eqnarray}\label{edmd1}
&&{\bf G}_\delta=\frac{1}{M}\sum_{m=1}^M \boldsymbol{\Psi}({ x}_m)^\top \boldsymbol{\Psi}({x_m + \delta x}_m)\nonumber\\
&&{\bf A}=\frac{1}{M}\sum_{m=1}^M \boldsymbol{\Psi}({ x}_m)^\top \boldsymbol{\Psi}({ x}_{m+1}),
\end{eqnarray}
with ${\bf K},{\bf G}_\delta,{\bf A}\in\mathbb{C}^{K\times K}$. The $\min-\max$ optimization problem (\ref{edmd_robust}) is in general nonconvex and will depend on the choice of dictionary functions. This is true because ${\cal F}$ in (\ref{edmd_robust}) is not in general concave function of $\delta$ for fixed $\bf K$. 
Hence, we convexify the problem as follows

\begin{equation}\label{edmd_robust_convex}
\min\limits_{\bf K}\max_{\delta{\bf G}\in \bar \Delta}\parallel ({\bf G}+\delta {\bf G}){\bf K}-{\bf A}\parallel_F
\end{equation}
where $\delta {\bf G}\in \mathbb{R}^{K\times K}$ is the new perturbation term characterized by uncertainty set  $\bar \Delta$ which lies in the feature space of dictionary function and the matrix ${\bf G}=\frac{1}{M}\sum_{m=1}^M \boldsymbol{\Psi}({ x}_m)^\top \boldsymbol{\Psi}({x}_m)
$. $\bar \Delta$ is the new uncertainty set defined in the feature space and will inherit the structure from set $\Delta $ in the data space. In particular, following result can be used to connect the two uncertainty set, $\Delta$ and $\bar \Delta$ when $\Delta$ is described as in Eq. (\ref{set1}).
\begin{proposition}\label{prop_uncertainty_set}
In the convex problem (\ref{edmd_robust_convex}), $\delta G$ is bounded as 
\begin{eqnarray}
\parallel \delta G\parallel_F \leq \lambda \Lambda \Gamma
\end{eqnarray}
where $\parallel\delta x_m\parallel_F\leq \lambda$, $\parallel {\bf \Psi}(x_m)\parallel_F\leq \Lambda$ and $\parallel {\bf \Psi}'(x_m)\parallel_F\leq \Gamma$ for all m.
\end{proposition}
\begin{proof}
From Taylor series expansion we have, ${\bf \Psi}(x_m+\delta x_m) = {\bf \Psi}(x_m) +  {\bf \Psi}'(x_m) \delta x_m$, where ${\bf \Psi}'(x_m)$ is the first derivative of ${\bf \Psi}(x)$ at $x_m$. Hence,
\begin{eqnarray*}
G_{\delta} = G + \frac{1}{M}\sum_{m=1}^M{\bf \Psi}^\top (x_m)\delta x_m {\bf \Psi}'(x_m) 
\end{eqnarray*}
Let $\delta G = \frac{1}{M}\sum_{m=1}^M{\bf \Psi}^\top (x_m)\delta x_m {\bf \Psi}'(x_m) $. Hence,
\begin{eqnarray*}
\parallel \delta G \parallel_F &=& \parallel \frac{1}{M}\sum_{m=1}^M{\bf \Psi}^\top (x_m)\delta x_m {\bf \Psi}'(x_m) \parallel_F\\
&\leq& \frac{1}{M}\sum_{m = 1}^M\parallel {\bf \Psi}^\top (x_m)\delta x_m {\bf \Psi}'(x_m)\parallel_F \\
&\leq&\frac{1}{M}\sum_{m=1}^M\parallel {\bf \Psi}^\top (x_m)\parallel_F \cdot \parallel\delta x_m \parallel_F \cdot \parallel {\bf \Psi}'(x_m)\parallel_F\\
&\leq& \lambda \Lambda \Gamma
\end{eqnarray*}
\end{proof}
Similar results can be used to connect the two uncertainty set $\Delta$ and $\bar \Delta$ for the case where $\Delta$ is described in Eqs. (\ref{set2}) and (\ref{set3}) to convert non-convex optimization problem to convex. For the uncertainty set given in Eq. (\ref{set2}), it can be easily shown that the problem simplifies to case described in the above proposition, where $\lambda$ depends on the eigenvalues of the matrices $Q_i$.  Similarly, for the uncertainty set given in (\ref{set3}), under the assumption that $\Delta$ is compact, it also boils down to proposition \ref{prop_uncertainty_set}. 

In \cite{Umesh_NSDMD}, we proposed Naturally Structured Dynamic Mode Decomposition (NSDMD) algorithm for finite dimensional approximation of the transfer Koopman and P-F operator. Apart from preserving positivity and Markov properties of the transfer operator, this algorithm exploits the duality between P-F and Koopman operator to provide the approximation of P-F operator. 
The algorithm presented for the robust approximation of Koopman operator can be combined with NSDMD for the robust approximation of P-F operator. In particular, under the assumption that all the dictionary functions are positive, following modification can be made to optimization formulation (\ref{edmd_robust_convex}) for the approximation of Koopman operator. 

\begin{eqnarray}\label{edmd_robust_convexPF} \nonumber
&&\min\limits_{\bf K}\max_{\delta{\bf G}\in \bar \Delta}\parallel ({\bf G}+\delta {\bf G}){\bf K}-{\bf A}\parallel_F\nonumber\\ \nonumber
&&{\rm s.t.}\quad{\bf K}_{ij}\geq 0\\ \nonumber
&&\qquad [\Lambda {\bf K}\Lambda^{-1}]_{ij}\geq 0\\
&&\qquad \Lambda{\bf K}\Lambda^{-1}\mathds{1}=\mathds{1}
\end{eqnarray}
where $\Lambda = \langle\boldsymbol{\Psi}(x),\boldsymbol{\Psi}(x)\rangle$ with $[\Lambda]_{ij}=\langle\psi_i,\psi_j\rangle$ is symmetric positive definite  matrix.
We refer the interested reader to \cite{Umesh_NSDMD} for details of NSDMD formulation.
Using duality the  robust approximation of the P-F operator, $\bf P$, can then be written as $\bf P={\bf K}^\top$. Most common approach for solving the robust optimization problem is by using a robust counterpart. In the following  section we show that the robust counterpart of the robust optimization problem can be constructed and is a convex optimization problem. 

\subsection{Robust Optimization, Regularization, and Sparsity}
The robust optimization problem (\ref{edmd_robust_convex}) has some interesting connection with optimization problems involving regularization term. In particular, we have following Theorem.

\begin{theorem}
Following two optimization problems 
\begin{eqnarray}
\min\limits_{\bf K}\max_{{\delta {\bf G}:}\parallel \delta{\bf G}\parallel_F\leq \lambda}\parallel ({\bf G}+\delta {\bf G}){\bf K}-{\bf A}\parallel_F \label{ocp}
\end{eqnarray}
\begin{eqnarray}\label{rob_eqv}
\min\limits_{\bf K}\parallel {\bf G}{\bf K}-{\bf A}\parallel_F+\lambda \parallel {\bf K}\parallel_F\label{regular}
\end{eqnarray}
are equivalent.
\end{theorem}
\begin{proof}
The following proof is from \cite{caramanis201214}, and we outline here for the convenience of the reader. 
For any $\delta \textbf{G} \in \bar \Delta$, such that $\parallel \delta \textbf{G} \parallel_F \leq \lambda$, $\parallel (\textbf{G} + \delta \textbf{G})\textbf{K} - \textbf{A}\parallel_F = \parallel(\textbf{GK} + \delta \textbf{GK} -\textbf{A}) \parallel_F$. By triangle inequality, using the fact that $\parallel \delta \textbf{G}\parallel_F \leq \lambda$, we have $\parallel (\textbf{G} + \delta \textbf{G})\textbf{K} - \textbf{A}\parallel_F \leq \parallel {\bf GK - A}\parallel_F + \lambda\parallel{\bf K}\parallel_F$. On the other hand, given any ${\bf K}$, we can choose a $\delta {\bf G}$, such that $\delta {\bf GK}$ is aligned with $({\bf GK} - {\bf A})$ and thus the two problems are equivalent. 
\end{proof}

The regularization term penalize the Frobenius norm of the matrix $\bf K$ and is also known as Tikhonov regularization. Equivalence to the more popular $\ell_1$ regularization or Lasso regularization can also be shown if we change the structure of the uncertainty set. Note that in the above theorem the uncertainty set $\bar \Delta$ is defined as follows:
\[\bar \Delta:=\{\delta {\bf G}\in \mathbb{R}^{K\times K}: \parallel {\delta \bf G}\parallel \leq \lambda\}\]
The equivalence between the robust optimization problem (\ref{edmd_robust_convex}) and the $\ell_1$ Lasso regularization can be established as follows.
\begin{theorem}
Define
\[\bar \Delta:=\{{\delta \bf G }=(\delta {\bf G}_1,\ldots, \delta{\bf G}_K)\in \mathbb{R}^{K\times K}: \parallel \delta{\bf G}_i\parallel_2\leq c\}.\]
Following two optimization problems are equivalent
\begin{eqnarray}
\min\limits_{\bf K}\max_{\delta {\bf G}\in \bar \Delta}\parallel ({\bf G}+\delta {\bf G}){\bf K}-{\bf A}\parallel_F
\end{eqnarray}
\begin{eqnarray}
\min\limits_{\bf K}\parallel {\bf G}{\bf K}-{\bf A}\parallel_F+c \sum_{k=1}^K\parallel {\bf K}_k\parallel_1\label{regular1}
\end{eqnarray}
where ${\bf K}_k $ is the $k^{th}$ column of matrix $\bf K$. 
\end{theorem}
Refer to \cite{caramanis201214} for the proof.
It is well known that the optimization problem with the regularization term especially the $\ell_1$ Lasso type regularization induce sparsity structure on the optimal solution. The equivalence between robust optimization and regularized optimization problem provides an alternate point of view to the sparsity structure, i.e., {\it robustness implies sparsity}. In \cite{jovanovic2012low}, sparse DMD algorithm is proposed where $\ell_1$ type regularization term is used to impose sparsity structure on the amplitude terms which appear in the temporal expansion of data along dynamic modes. We expect similar robust optimization-based viewpoint can be provided to the sparse DMD algorithm proposed in \cite{jovanovic2012low}.

\section{Design of robust predictor}\label{section_predictor}

While the optimization problem with the regularization term (\ref{regular1}) and the robust optimization problem (\ref{ocp}) are equivalent, the robust optimization point of view to the optimization problem offers a particular advantage. In particular, robust optimization viewpoint provides a systematic way of determining the regularization parameter which often is a tuning parameter. On the other hand, the optimization formulation with the regularization term has interesting interpretation borrowed from machine learning literature.  In problems involving model fitting from data one of the fundamental tradeoff arise between the quality of approximation and complexity of approximation function. For example, in the absence of the regularization term, the optimization problem will find a matrix $\bf K$ that best tries to fit the training data to the model.  Hence, without the regularization term, there will be a tendency to over-fit model parameters to data. Such over-fitted model will perform well on the training data and give a smaller value of optimal cost. However, these over-fitted models will often perform poorly on the test data-set. On the other hand,  the optimization problem with the regularization term tries to strike a balance between over fitting and prediction. This observation on the role of regularization term has an important implication on the proposed application of transfer operator framework for the design of data-driven predictor. For the design of predictor dynamics, we first use training data for the approximation of the transfer Koopman operator. Let $\{x_0,\ldots,x_M\}$ be the training data-set and $\bf K$  be the finite-dimensional approximation of the transfer Koopman operator obtained using the robust algorithm (\ref{edmd_robust_convex}). Let $\bar x_0$ be the initial condition from which we want to predict the future. The initial condition from state space is mapped to the feature space using the same choice of basis function used in the robust approximation of Koopman operator i.e., \[\bar x_0\implies {\bf \Psi}(\bar x_0)^\top=: {\bf z}\in \mathbb{R}^K.\] This initial condition is propagated using Koopman operator as \[{\bf z}_n={\bf K}^n{\bf z}.\]
The predicted trajectory in the state space is then obtained as 
\[\bar x_n=C {\bf z}_n\]
where matrix $C$ is obtained as the solution of the following least square problem
\begin{eqnarray}\label{C_pred}
\min_C\sum_{i = 1}^M \parallel x_i - C \boldsymbol \Psi (x_i)\parallel_2^2
\end{eqnarray}
In the simulation section, we demonstrate the effectiveness of the proposed robust prediction algorithm on linear and nonlinear systems.

\section{Simulation Results}\label{section_simulation}
In this simulation section we compare our proposed robust optimization-based approach for the approximation of transfer operator with existing approaches namely subspace DMD, and regular DMD. Hence for the purpose of completeness we explain in brief the subspace DMD method for approximation from  \cite{PhysRevE.96.033310}. 

Consider the random dynamical system (\ref{system}) with noisy observables $h : X\times S \to \mathbb{C}^K$ :
\begin{eqnarray}
h(x_t,s_t) = {\bf \Psi}(x_t) + w(s_t)
\end{eqnarray} 
where $w:S\to\mathbb{C}^K$ is a random variable on a probability space $(S,\Sigma_S,\mu_S)$ of the observation noise.
Consider the data matrix as a concatenation of $m$ observations as
\begin{eqnarray}
Y_t = [\begin{matrix}
h(x_t) & \cdots & h(x_{t+m-1})
\end{matrix}]
\end{eqnarray}
%
Let $(Y_0,Y_1,Y_2,Y_3)$ be a quadruple formed from the data set, such that 
\begin{eqnarray*}
&&Y_0 = [h(x_0)\quad h(x_1)\quad \cdots \quad h(x_{m-1})]\\
&&Y_1 = [h(x_1)\quad h(x_2)\quad \cdots\quad h(x_{m})]\\
&&Y_2 = [h(x_2)\quad h(x_3)\quad\cdots \quad h(x_{m+1})]\\
&&Y_3 = [h(x_3)\quad h(x_4)\quad \cdots \quad h(x_{m+2})]\\
\end{eqnarray*}

Define $Y_p, Y_f$ as 
\begin{eqnarray}
Y_p = [Y_0^\top \quad Y_1^\top]^\top ; \quad Y_f = [Y_2^\top \quad Y_3^\top]^\top
\end{eqnarray}

Then the Subspace DMD algorithm \cite{PhysRevE.96.033310} is as follows :
\begin{enumerate}
\item{Build the data sets $Y_p$ and $Y_f$ from the data set.}
\item{Compute the orthogonal projection of rows of $Y_f$ onto the row space of $Y_p$ as $O = Y_f\mathbb{P}_{Y_p^{\mathsf{H}}}$.}
\item{Compute the compact SVD $O = U_qS_qV_q^{\mathsf{H}}$ and define $U_{q_1}$ and $U_{q_2}$ by the first and last $K$ rows of $U_q$ respectively.}
\item{Compute the compact SVD $U_{q_1}=USV^{\mathsf{H}}$ and define $\tilde{A}=U^{\mathsf{H}}Y_fVS^{-1}$.}
\item{Compute the eigenvalues $\lambda$ and  eigenvectors $\tilde{w}$ of $\tilde{A}$.}
\item{Return the dynamic modes $w=\lambda^{-1}U_{q_2}VS^{-1}\tilde{w}$ and corresponding eigenvalues $\lambda$.}
\end{enumerate}

\subsection{IEEE 9 bus system}
In this section, we consider the IEEE 9 bus system, the line diagram of which is shown in Fig. \ref{9_bus_fig}. The model used is based on the modeling described in \cite{Sauer_pai_book}.

\begin{figure}[htp!]
\centering
\includegraphics[scale=.5]{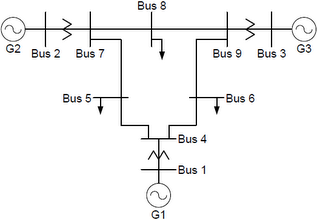}
\caption{IEEE 39 bus system.}\label{9_bus_fig}
\end{figure}
The power network is described by a set of differential algebraic equations (DAE) and the power system dynamics is divided into three parts: differential equation model describing the generator and load dynamics, algebraic equations at the stator of the generator and algebraic equations describing the network power flow. We consider a power system model with $n_g$ generator buses and $n_l$ load buses. The generator dynamics at each generator bus can be represented as a $4^{th}$ order dynamical model:

{\small
\begin{eqnarray*}\label{generator_dynamic_eq}
\begin{aligned}
&\frac{d\delta_i}{dt}  = \omega_i - \omega_s \\
&\frac{d\omega_i}{dt}  = \frac{T_{m_i}}{M_i} - \frac{E_{q_i}^{\prime} I_{q_i}}{M_i} - \frac{(X_{q_i} - X_{d_i}^{\prime})}{M_i} I_{d_i} I_{q_i} - \frac{D_i (\omega_i-\omega_s)}{M_i} \\
&\frac{d E_{q_i}^{\prime}}{dt}  = -\frac{E_{q_i}^{\prime}}{T_{{do}_i}^{\prime}} - \frac{(X_{d_i} - X_{d_i}^{\prime})}{T_{{do}_i}^{\prime}} I_{d_i} + \frac{E_{{fd}_i}}{T_{{do}_i}^{\prime}} \\
&\frac{dE_{{fd}_i}}{dt}  = -\frac{E_{{fd}_i}}{T_{A_i}} + \frac{K_{A_i}}{T_{A_i}} (V_{{ref}_i} - V_i) 
\end{aligned}
\end{eqnarray*}}
The algebraic equations at the stator of the generator are: 
{\small
\begin{align}\label{stator_algebraic_eq}
\begin{split}
& V_i \sin(\delta_i - \theta_i) + R_{s_i} I_{d_i} - X_{q_i} I_{q_i}  = 0 \\
& E_{q_i}^{\prime} - V_i \cos(\delta_i - \theta_i) - R_{s_i} I_{q_i} - X_{d_i}^{\prime} I_{d_i}  = 0 \\
& \qquad \text{for}\quad  i = 1,\dots, n_g.
\end{split}
\end{align}}
The network equations corresponding to the real and reactive power at generator and load buses are shown below. 
{\small
\begin{eqnarray}\label{network_eq}
\begin{aligned}
& I_{d_i} V_i \sin(\delta_i - \theta_i) + I_{q_i} V_i \cos(\delta_i - \theta_i) + P_{L_i} (V_i) \\ 
& - \sum_{k=1}^{\bar n} V_i V_k Y_{ik} \cos(\theta_i - \theta_k - \alpha_{ik}) = 0 \\ 
&  I_{d_i} V_i \cos(\delta_i - \theta_i) - I_{q_i} V_i \sin(\delta_i - \theta_i) + Q_{L_i} (V_i) \\ 
& - \sum_{k=1}^{\bar n} V_i V_k Y_{ik} \sin(\theta_i - \theta_k - \alpha_{ik}) = 0 \\
& \text{for} \;\ i = 1,\dots, n_g. \\
& P_{L_i}(V_i) - \sum_{k=1}^{\bar n} V_i V_k Y_{ik} \cos(\theta_i - \theta_k - \alpha_{ik}) = 0 \\
& Q_{L_i}(V_i) - \sum_{k=1}^{\bar n} V_i V_k Y_{ik} \sin(\theta_i - \theta_k - \alpha_{ik}) = 0 \\
& \text{for} \;\ i = n_g +1,\dots, n_g + n_l. 
\end{aligned}
\end{eqnarray}}
here, $\delta_i$, $\omega_i, E_{q_i}$, and $E_{{fd}_i}$ are the dynamic states of the generator and correspond to the generator rotor angle, the angular velocity of the rotor, the quadrature-axis induced emf and the emf of fast acting exciter connected to the generator respectively. The algebraic states $I_{d_i}$ and $I_{q_i}$ are the direct-axis and quadrature-axis currents induced in the generator respectively. Each bus voltage and its angle are denoted by $V_i$ and $\theta_i$. The parameters $T_{m_i}, V_{{ref}_i}, \omega_s, M_i$, and $D_i$ are the mechanical input and machine parameters applied to the generator shaft, reference voltage, rated synchronous speed, generator inertia, and internal damping. 
The stator internal resistance is denoted by $R_{s_i}$ and $X_{q_i}$, $X_{d_i}$, $X_{d_i}^{\prime}$ are the quadrature-axis salient reactance, direct-axis salient reactance and direct-axis transient reactance. The exciter gain and time-constant are given by $K_{A_i}$ and $T_{A_i}$.

A power system stabilizer (PSS), that acts as a local controller to the generator is designed based on the linearized DAEs. The input to the PSS controller is $ \omega_i(t)$ and PSS output, $V_{{ref}_i}(t)$, is fed to the fast acting exciter of the generator. An IEEE Type-I PSS is considered here which consists of a wash-out filter and two phase-lead filters. The transfer function of PSS is given by

\begin{align}
\frac{\Delta V_{{ref}_i}(s)}{\Delta \omega_i(s)} = k_{pss} \frac{(1+sT_{num})^2}{(1+sT_{den})^2} \frac{s T_w}{1+sT_w} 
\label{eq_pss}
\end{align} 
where $k_{pss}$ is the PSS gain, $T_w$ is the time constant of wash-out filter and $T_{num}, T_{den}$ are time constants of phase-lead filter with $T_{num} > T_{den}$. 

Elimination of the algebraic variables by Kron reduction, generates a reduced order dynamic model given by $\Delta \dot{x}_g  = {A_{gg}} \Delta x_g + E_1 \Delta \tilde{u}$ where $\Delta {x}_g \in \mathbb{R}^{{7n_g}}$ and $\Delta \tilde{u} \in \mathbb{R}^{{n_g}}$. 

In this example, there are three generators and we consider static loads and hence the dynamic model is of dimension 21. Data was collected over 25 time steps, with sampling time $\delta t = .2$ seconds. The obtained data was corrupted with uniform noise with support on $[-0.4,0.4]$. 
\begin{figure}[htp!]
\centering
\subfigure[]{\includegraphics[scale=.3]{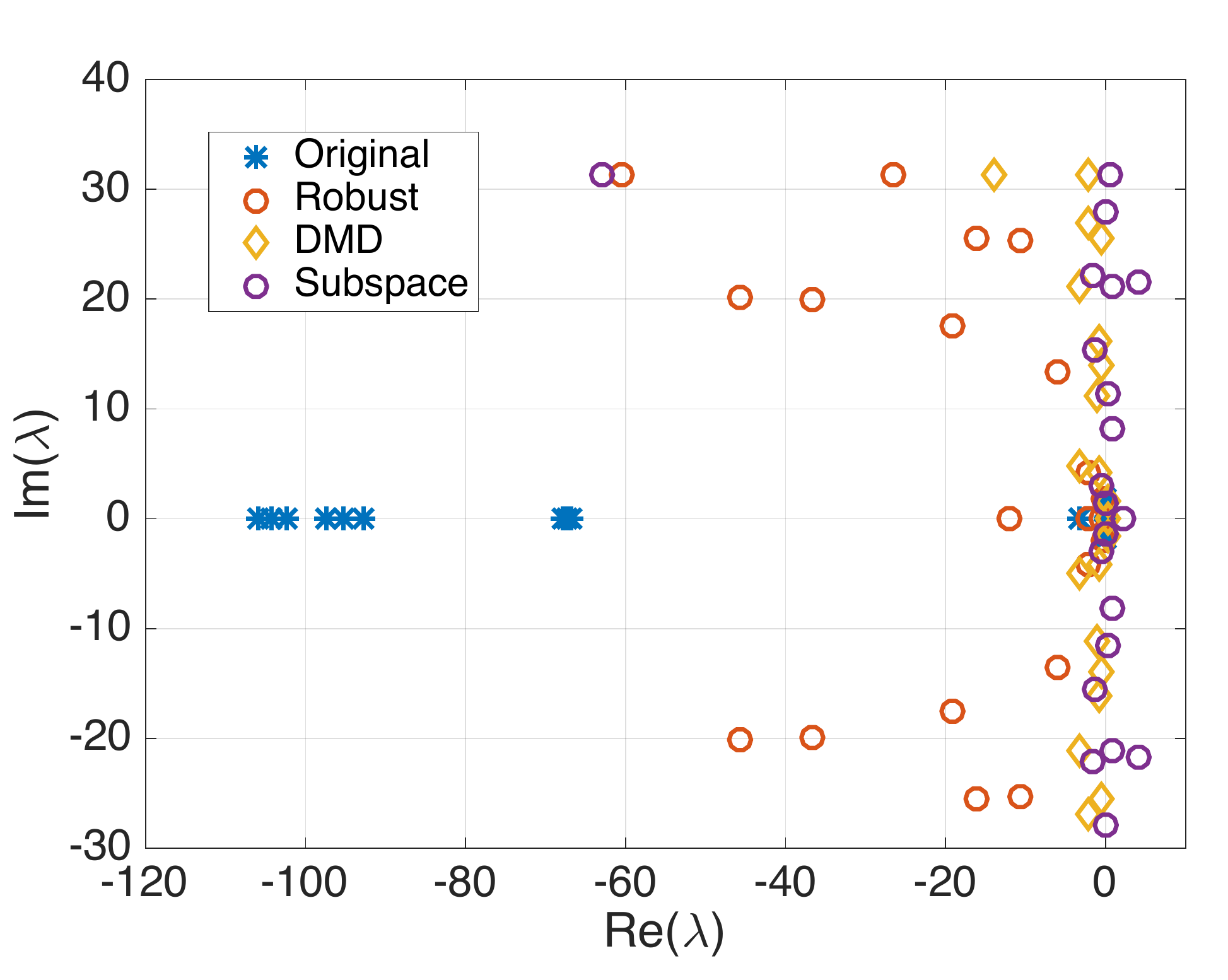}}
\subfigure[]{\includegraphics[scale=.3]{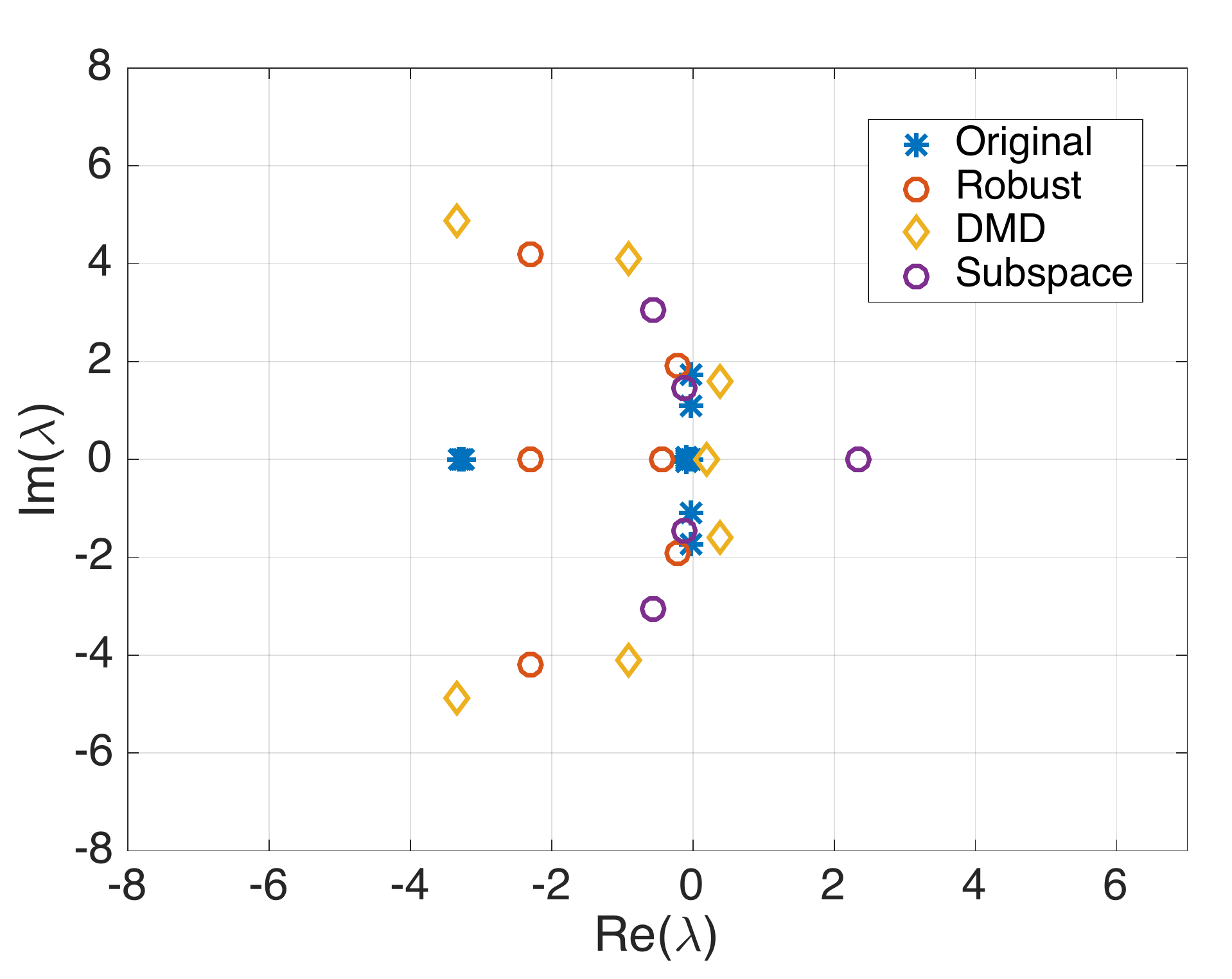}}
\caption{(a) Eigenvalues obtained using Ribust DMD, normal DMD and Subspace DMD. (b) Dominant eigenvalues.}\label{fig_eig_9_bus}
\end{figure}

In Fig. \ref{fig_eig_9_bus}(a) we compare the eigenvalues obtained using different methods, namely, Robust DMD, normal DMD and subspace DMD. In Fig. \ref{fig_eig_9_bus}(b) we show the dominant eigenvalues and find that the robust DMD performs much better, compared to other methods, and captures the dominant eigenvalues with more accuracy. Moreover, normal DMD and subspace DMD generates eigenvalues which are placed on the right half plane, implying the system to be unstable, though the original system is stable. Hence, we conclude that Robust DMD does a better job compared to the existing methods. In the above comparison, we had chosen linear dictionary functions. However, one can choose radial basis functions or Hermite polynomials as dictionary functions and can get similar results.

\subsection{Noisy rotations on a circle}
We consider the example discussed in \cite{mezic_stochastic_koopman_spectrum},\cite{junge2004uncertainty}. A dynamical system $T$, corresponding to rotation on the unit circle $S^1$ is given by 
\begin{equation}
T(x) = x + \theta
\end{equation}
where $\theta\in S^1$ is a constant number. We consider a stochastic dynamical system with process noise, so that the RDS is given by
\begin{equation}
T(x,\xi) = x + \theta + \xi
\end{equation}
where $\xi$ are independent and identically distributed random variables taking values in $[-0.7,0.7]$. In this example we considered $x(0)=1$, $\theta = \pi/320$. We consider the dictionary functions as 
\begin{equation}
\mathbf{\Psi} = e^{2\pi i n x}
\end{equation}
where $n = -50,-49, \cdots , 49, 50$. Data was collected for $6000$ time steps and data for the first $50$ time steps was used for training purpose. 

\begin{figure}[htp!]
\centering
\subfigure[]{\includegraphics[scale=.255]{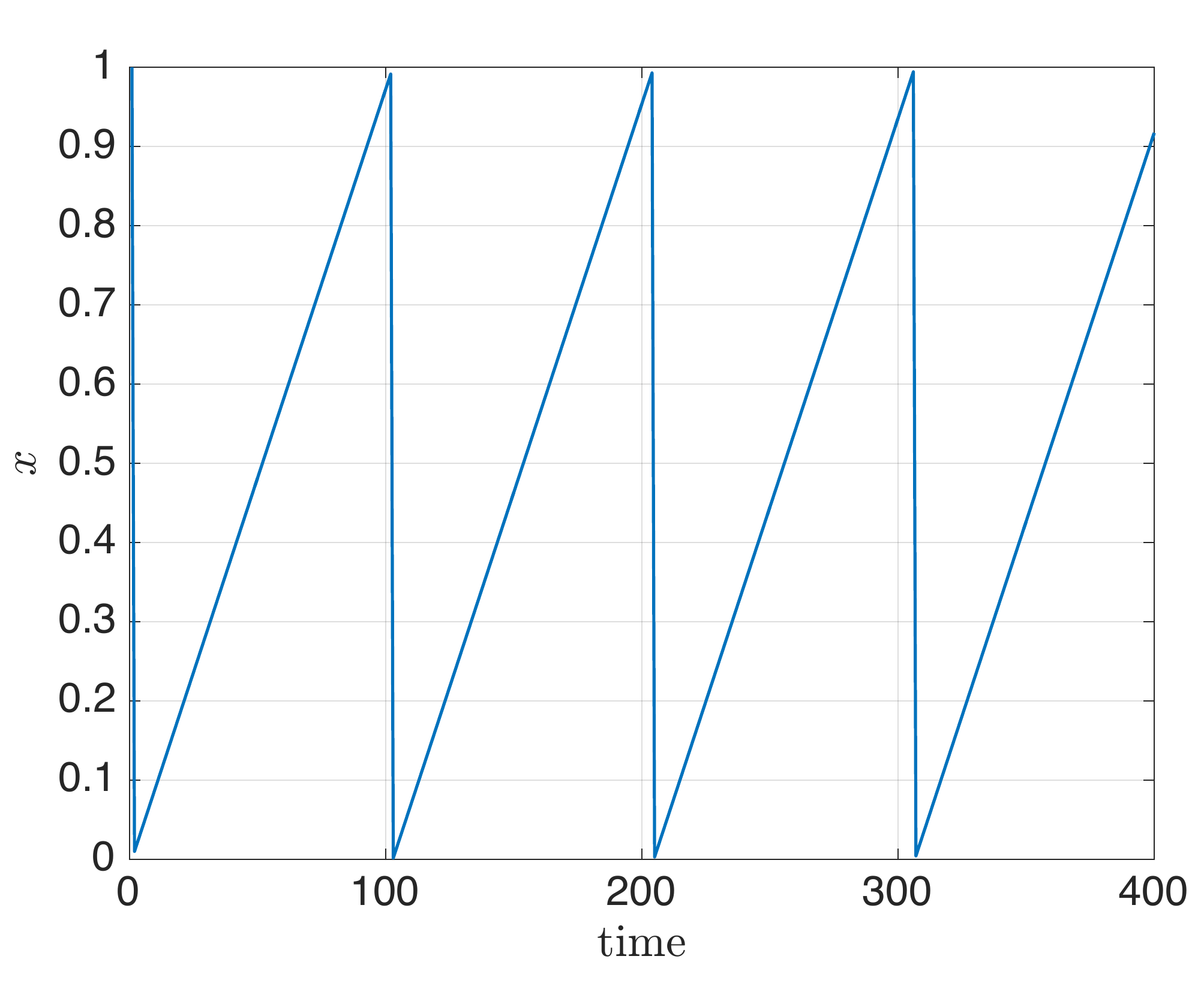}}
\subfigure[]{\includegraphics[scale=.25]{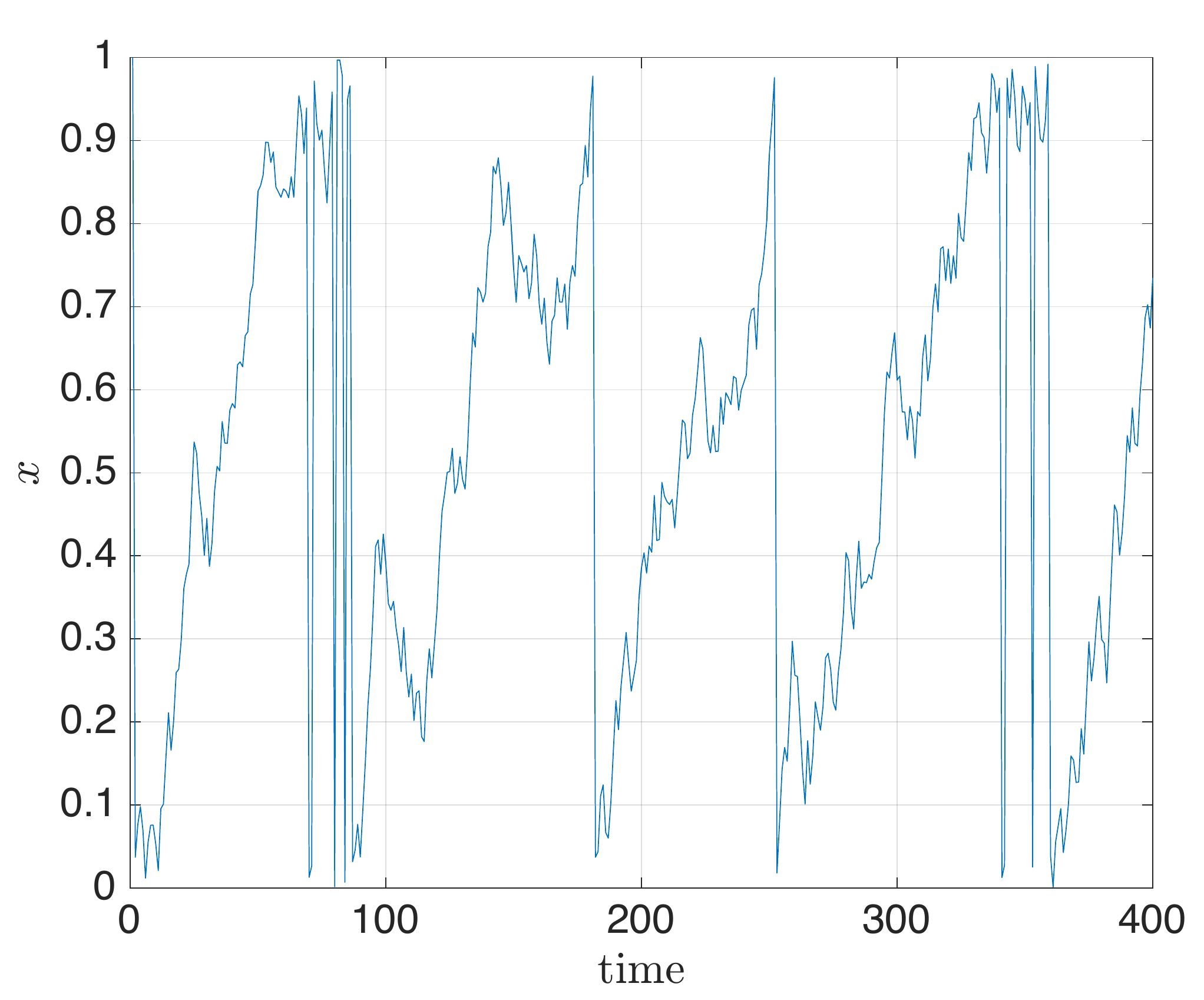}}
\subfigure[]{\includegraphics[scale=.28]{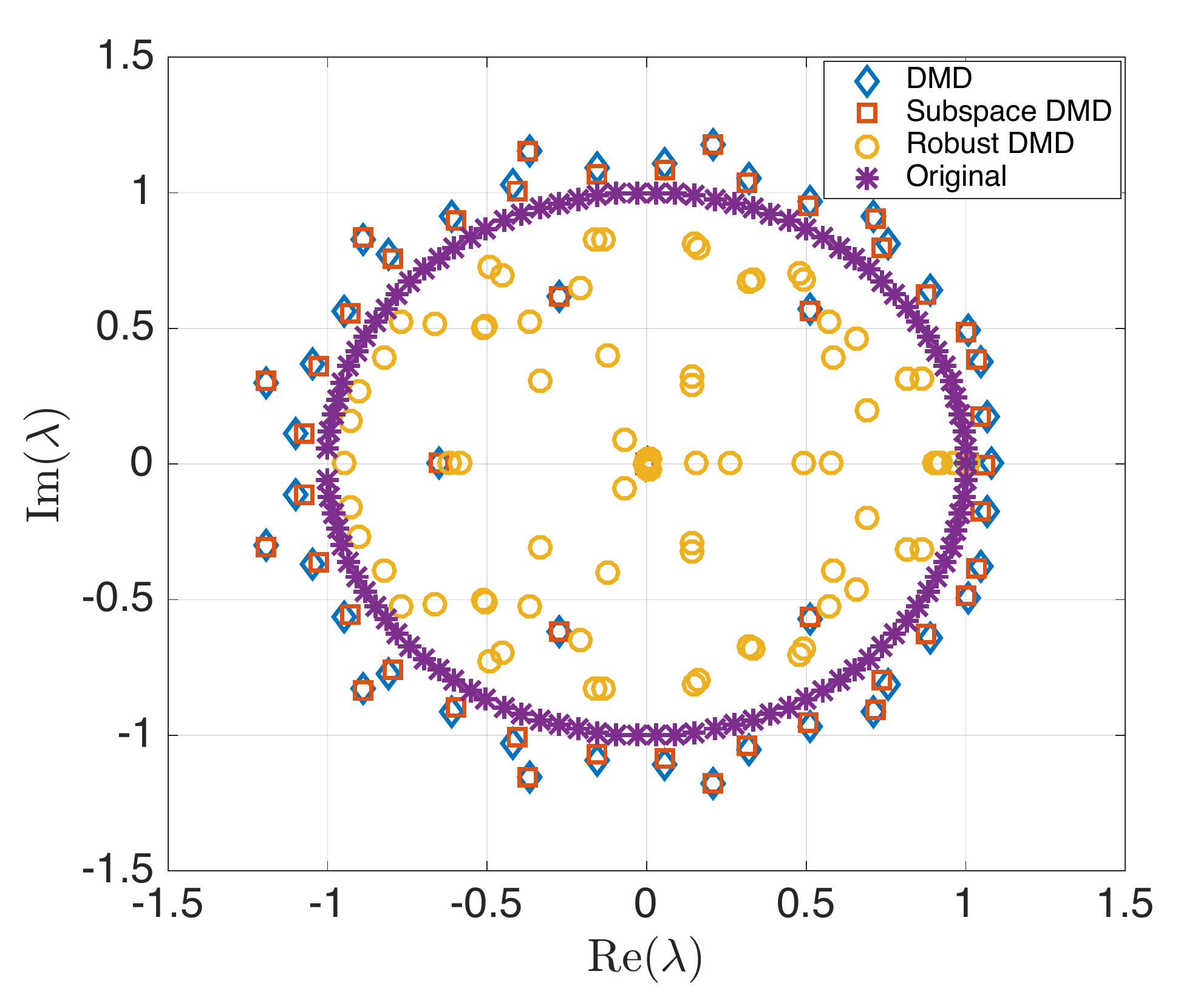}}
\caption{(a) Clean data. (b) Noisy data. (c) Comparison of eigenvalues obtained with different algorithms.}\label{fig_rotation}
\end{figure}
The clean data and the noisy data is shown in Fig. \ref{fig_rotation}(a) and (b) respectively. The eigenvalues obtained using the different algorithms is plotted in Fig. \ref{fig_rotation}(c) and it can be seen that the robust DMD algorithm provides closer match for the eigenvalues of the deterministic system. Moreover, normal DMD and subspace DMD yields unstable eigenvalues, though the original system is not unstable. As mentioned earlier, normal DMD and subspace DMD works for RDS when the training data set is large enough, but in reality it may not be possible to obtain such large data sets and in that case robust DMD provides much better approximation of the Koopman operator.

\subsection{Stuart-Landau equation}
The nonlinear stochastic Stuart-Landau system \cite{PhysRevE.96.033310}. The stochastic Stuart-Landau equation on a complex function $z(t) = r(t)\exp (i\theta(t))$ is given by 
\begin{eqnarray}\label{stu_lan}
\dot{z} = (\mu + i\gamma)z - (1+ i\beta)|z|^2z + \sigma\xi(t),
\end{eqnarray}
where $\xi(t)$ is a white Gaussian noise of unit variance and $i$ is the imaginary unit. In absence of process noise, the solution of (\ref{stu_lan}) evolves on the limit cycle $|z| = \sqrt{\mu}$. Hence, the continuous time eigenvalues lie on the imaginary axis, if the process noise is absent. The discretized version of (\ref{stu_lan}) is

{\small
\begin{eqnarray}\label{stu_lan_dis}
\begin{pmatrix}
r_{t+1}\\
\theta_{t+1}
\end{pmatrix} = \begin{pmatrix}
r_t + (\mu r_t -r_t^3)\delta t\\
\theta_t + (\gamma - \beta r_t^2)\delta t
\end{pmatrix} + \sigma_p\begin{pmatrix}
\delta t & 0\\
0 & \delta t/ r_t
\end{pmatrix}\xi_t.
\end{eqnarray}
}
We assume that the observation are corrupted with noise and of the form
\begin{eqnarray}
y_t = \begin{pmatrix}
e^{-10i\theta_t} & e^{-9i\theta_t} & \cdots & e^{9i\theta_t} & e^{10i\theta_t}
\end{pmatrix} + \sigma_o w_t.
\end{eqnarray}
The noisy output data was used to construct the finite dimensional approximation of the operator, with $\delta t = 0.01$. Both the process noise and measuement noise considered here are uniform with support $[-.03,0.3]$ and $[-0.1,0.1]$ respectively. A sample trajectory is shown in Fig. \ref{data_nonlinear}.

\begin{figure}[htp!]
\centering
\includegraphics[scale=.275]{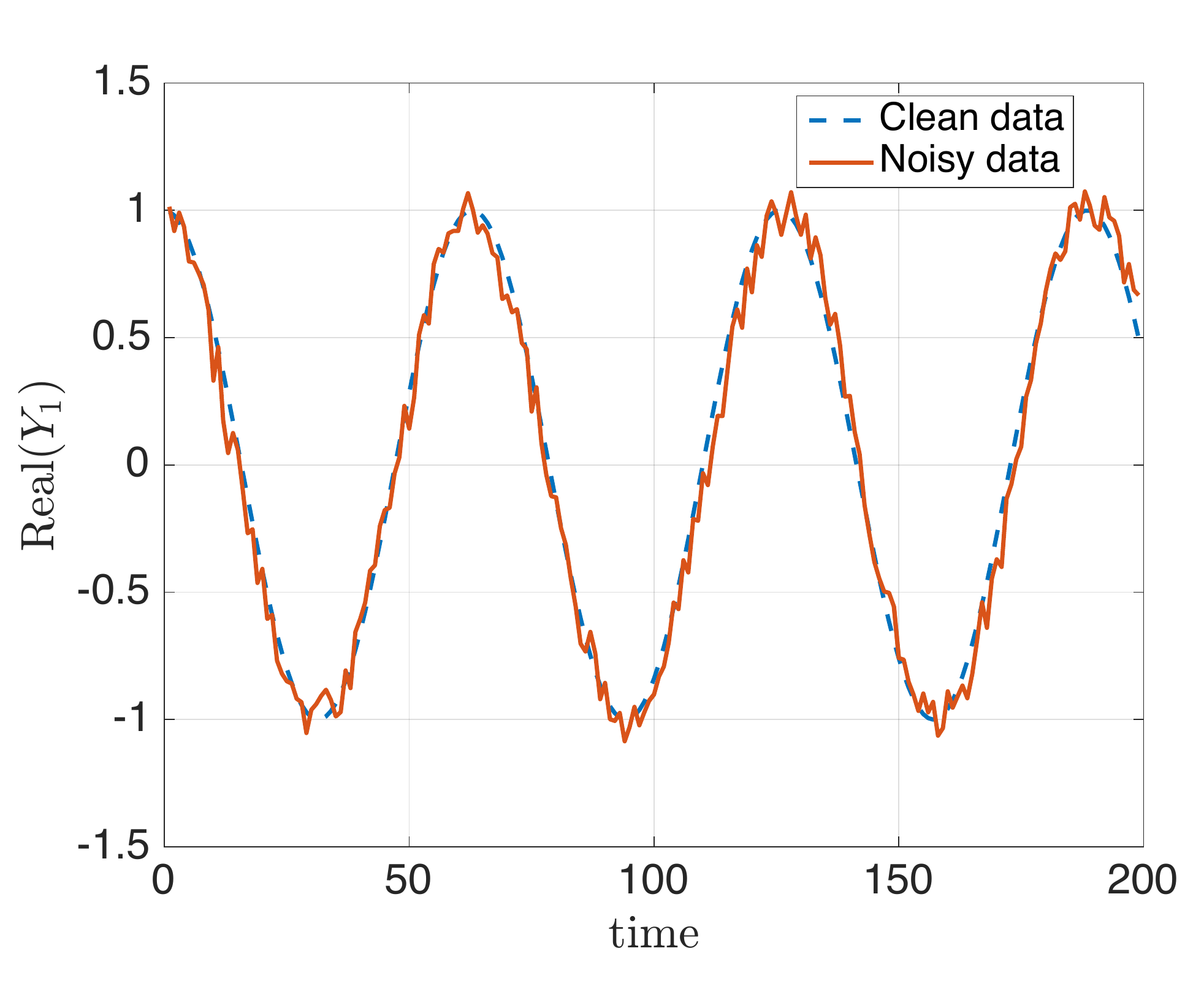}
\caption{Sample output trajectory}\label{data_nonlinear}
\end{figure}
\begin{figure}[htp!]
\centering
\subfigure[]{\includegraphics[scale=.275]{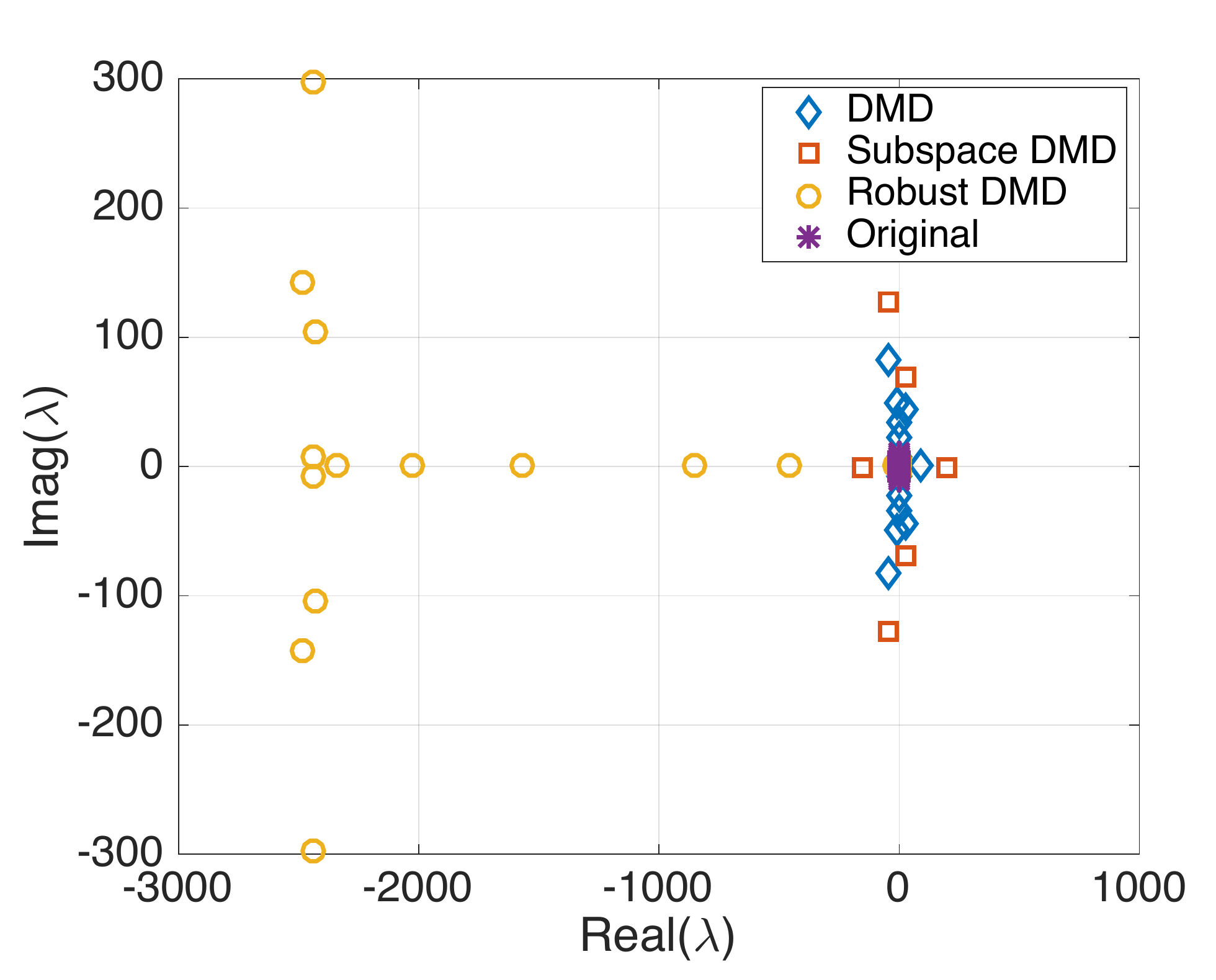}}
\subfigure[]{\includegraphics[scale=.275]{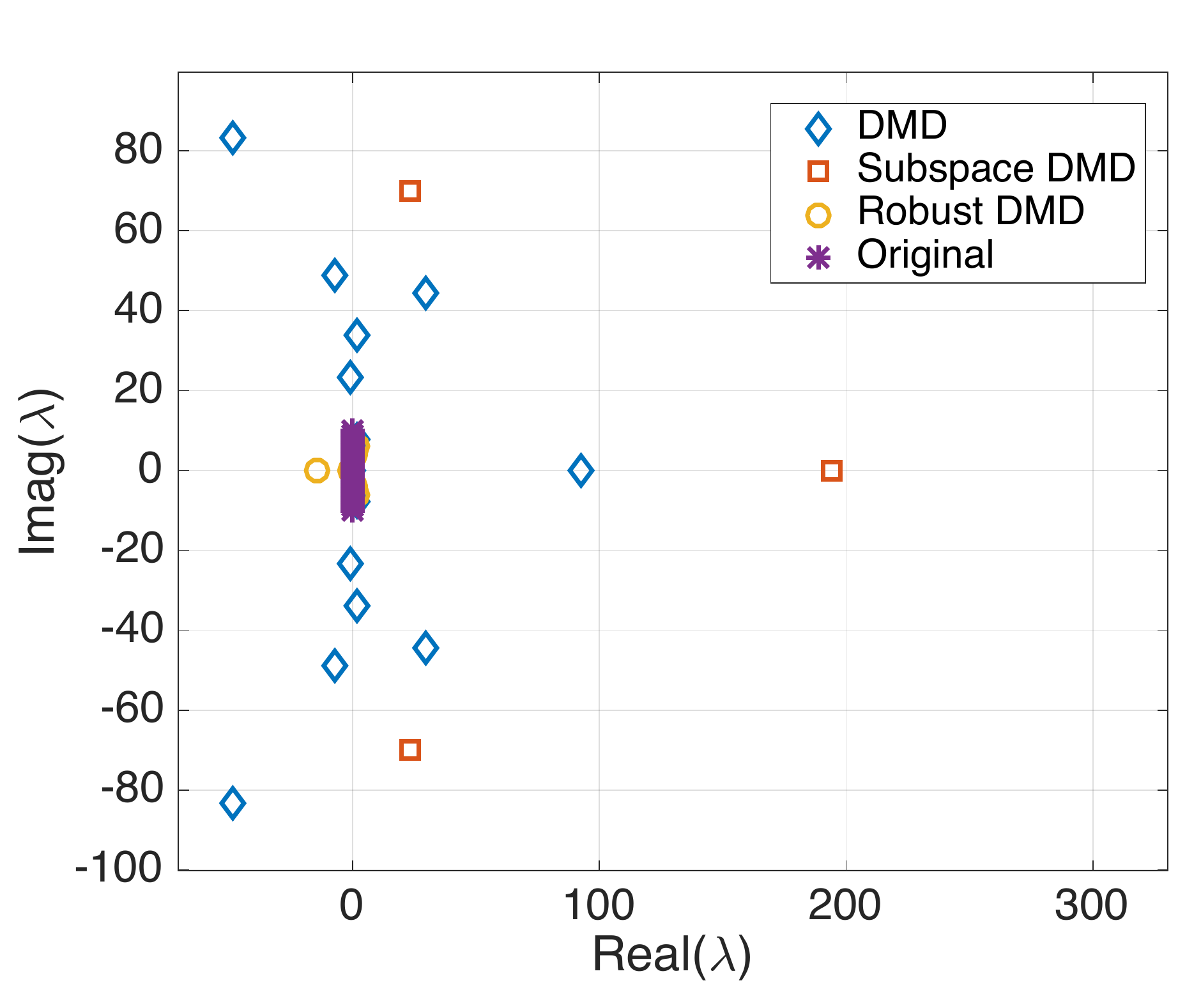}}
\caption{(a) Eigenvalue comparison of Robust DMD, normal DMD and subspace DMD. (b) Dominant eigenvalues.}\label{eig_nonlinear}
\end{figure}

\begin{figure}[htp!]
\centering
\subfigure[]{\includegraphics[scale=.275]{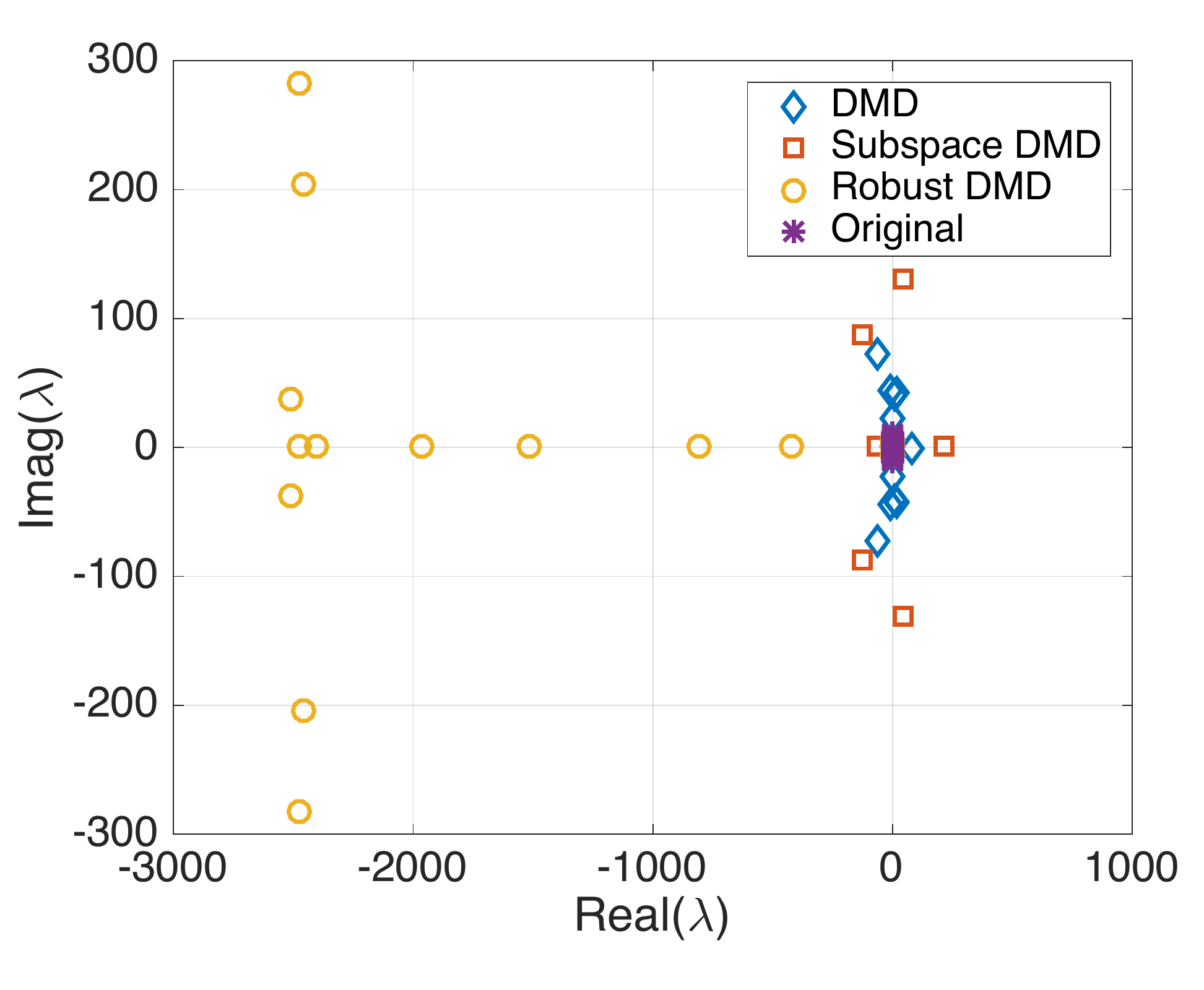}}
\subfigure[]{\includegraphics[scale=.278]{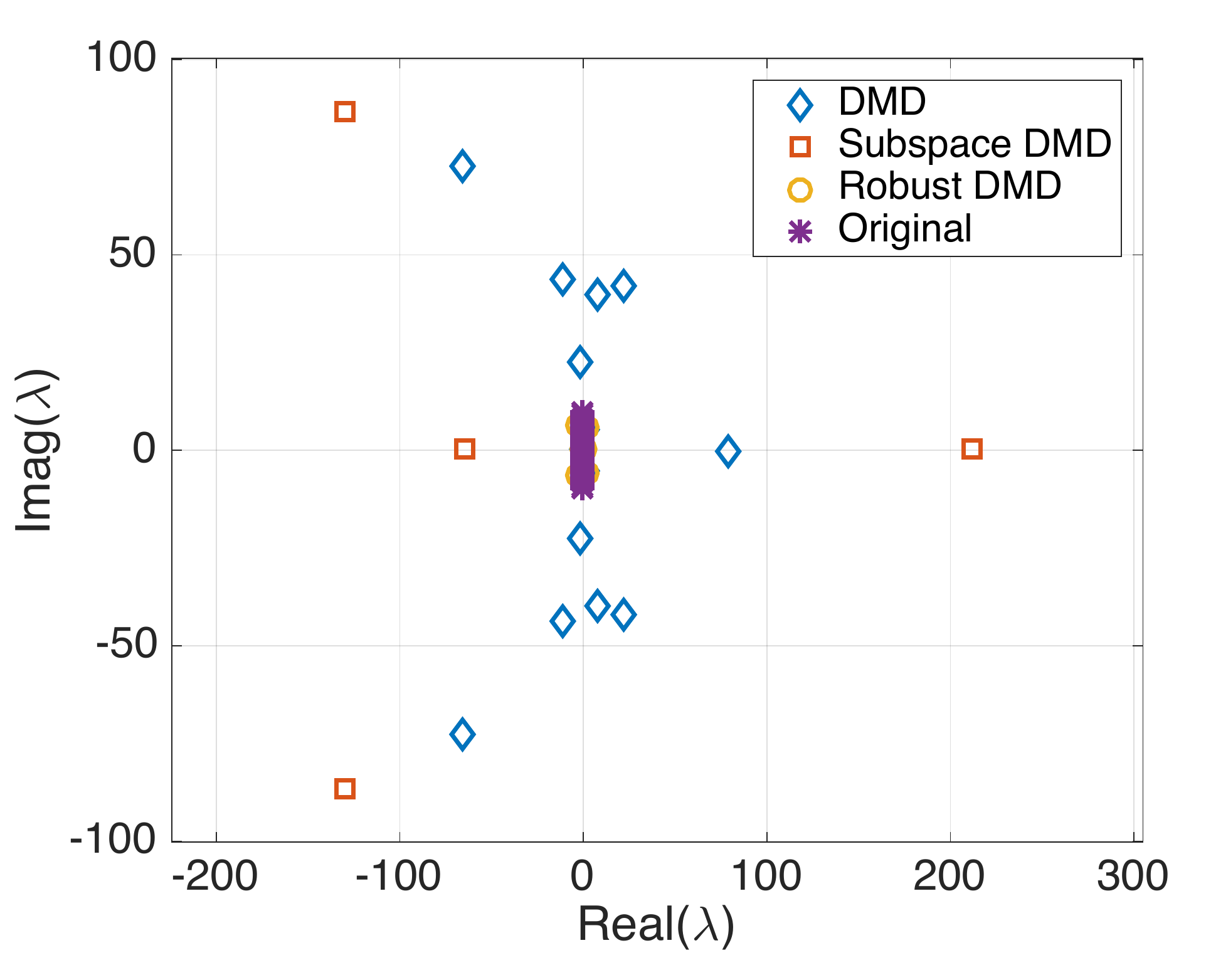}}
\caption{(a) Eigenvalue comparison of Robust DMD, normal DMD and subspace DMD. (b) Dominant eigenvalues.}\label{eig_nonlinear_point9}
\end{figure} 

The system was initialized at $(1,-\pi)$ and data was obtained for 100 time steps and data of the first 30 steps was used as the training data and the prediction was made over future steps.  The dictionary functions we chose are $e^{in\theta_t}$, $n=-10,-9,\cdots , 9,10$.  The eigenvalues of the approximate robust DMD Koopman operator, normal DMD Koopman and subspace DMD operator for the initial condition $(1,-\pi)$ are shown in Fig. \ref{eig_nonlinear}(a). It is observed that all the algorithms are capturing the oscillatory nature of the underlying system, but normal DMD and subspace DMD generate a Koopman operator which has an eigenvalue on the right half plane implying the system to be unstable. The first few dominant eigenvalues obtained using the three different algorithms are shown in Fig. \ref{eig_nonlinear}(b).  

\begin{figure}[htp!]
\centering
\subfigure[]{\includegraphics[scale=.275]{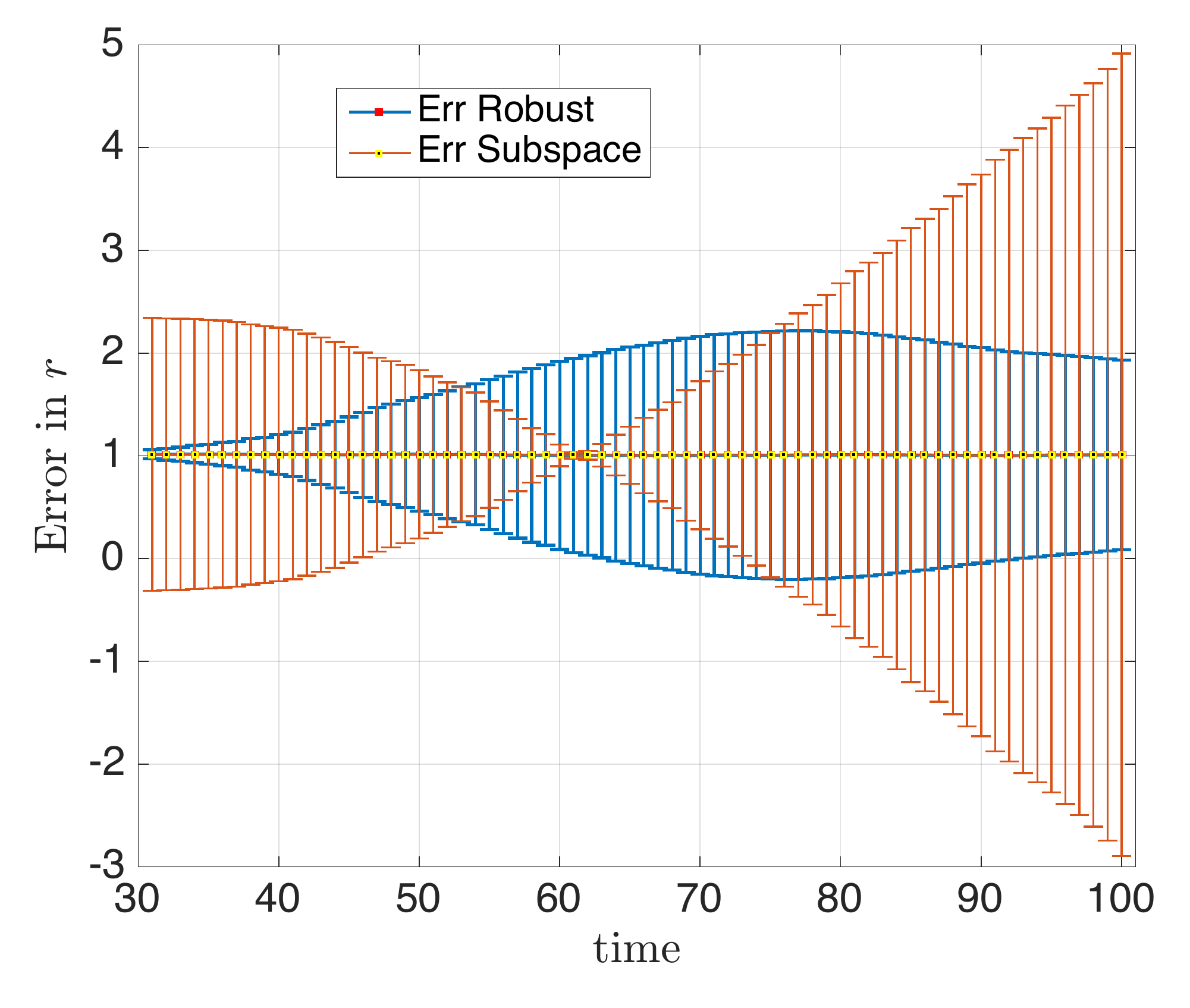}}
\subfigure[]{\includegraphics[scale=.275]{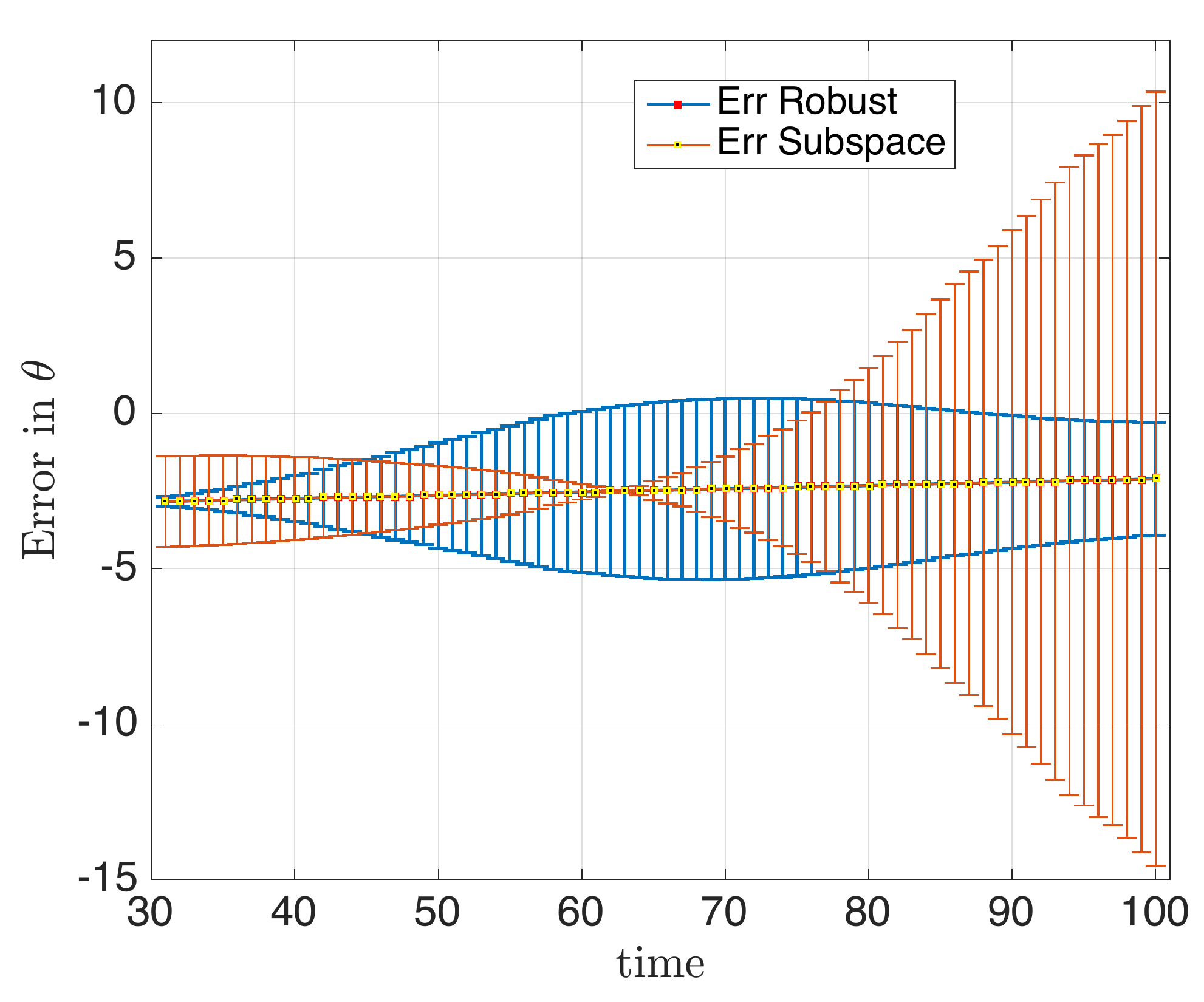}}
\caption{(a) Error in prediction in $r$ for initial state at $(1,-\pi)$. (b) Error in prediction in $\theta$ initial state at $(1,-\pi)$.}\label{err_nonlinear}
\end{figure}

Once the Koopman operators are obtained, one can use the obtained operator to predict the future. In Fig. \ref{err_nonlinear}(a) and (b), we compare the errors in prediction in $r$ and $\theta$ respectively, when using robust DMD and subspace DMD for the system starting from $(1,-\pi)$. The errors are plotted around the actual values of $r$ and $\theta$. For example, the error in $r$ is plotted around the actual value of $r=1$.  
As mentioned earlier, the first 30 time step data was used as the training data, and the operators obtained using the training data has been used to predict future 70 time steps. As can be seen from Fig. \ref{err_nonlinear}(a) and (b), robust DMD algorithm yields much smaller error, when compared to subspace DMD. The error in prediction using normal DMD increases exponentially and is not shown in the error plot. However, we find that the error using subspace DMD decreases initially and then grows exponentially. This is because the coupling from the unstable subspace to the stable subspace is small and hence initially the dominant stable eigenvalues makes the error decrease but as the system evolves, the effect of the unstable eigenvalue becomes more prominent and the error grows exponentially.



As mentioned earlier, compared to subspace DMD Robust DMD provides better approximation of systems, when the size of training data is not substantial. To illustrate this point, we used different sizes of training data set and used it to predict future 10 time steps. In particular, we varied the training data from 10 time steps to 40 time steps and looked at the average error in prediction in both $r$ and $\theta$ over 10 future time steps.  

\begin{figure}[htp!]
\centering
\subfigure[]{\includegraphics[scale=.33]{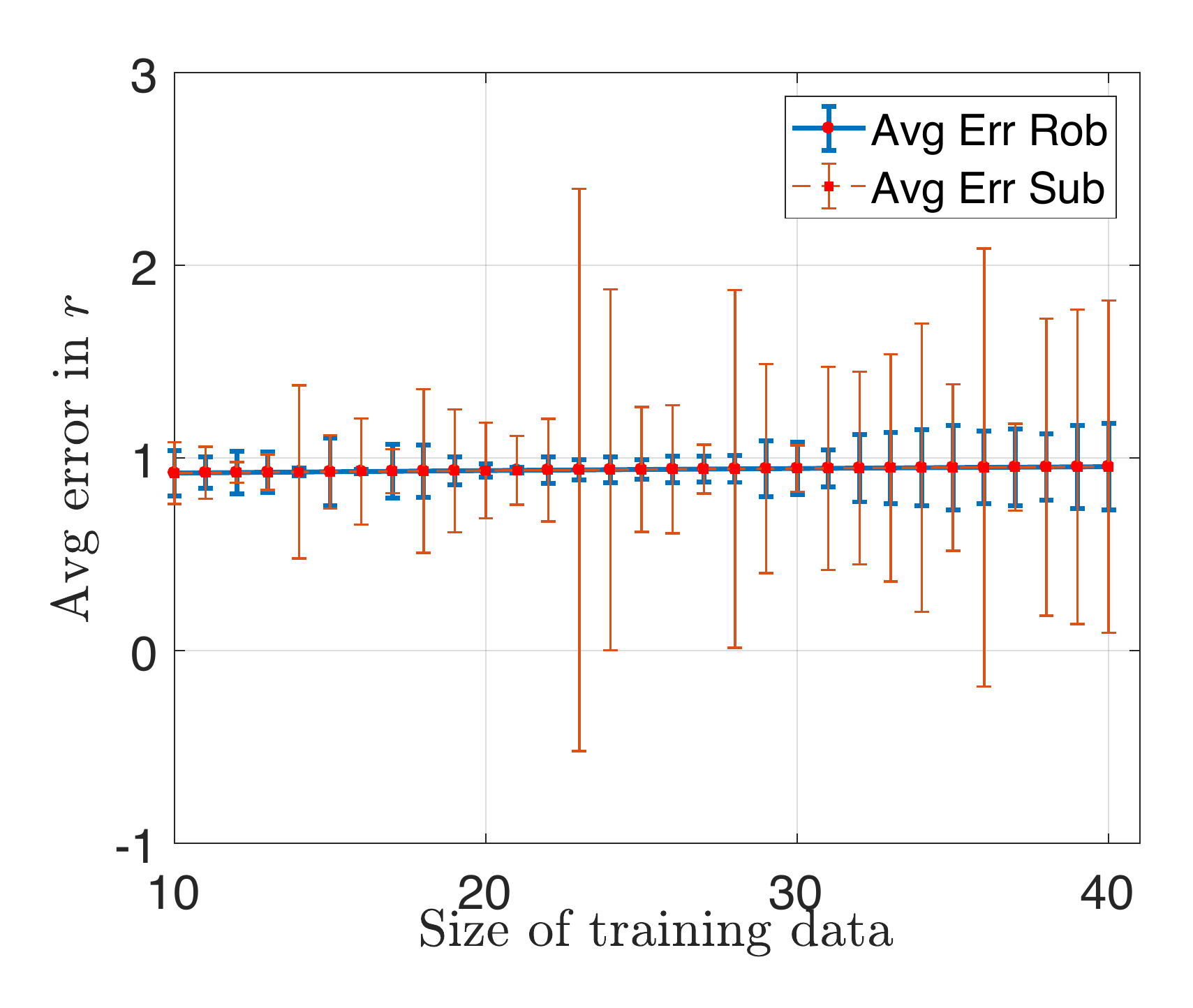}}
\subfigure[]{\includegraphics[scale=.33]{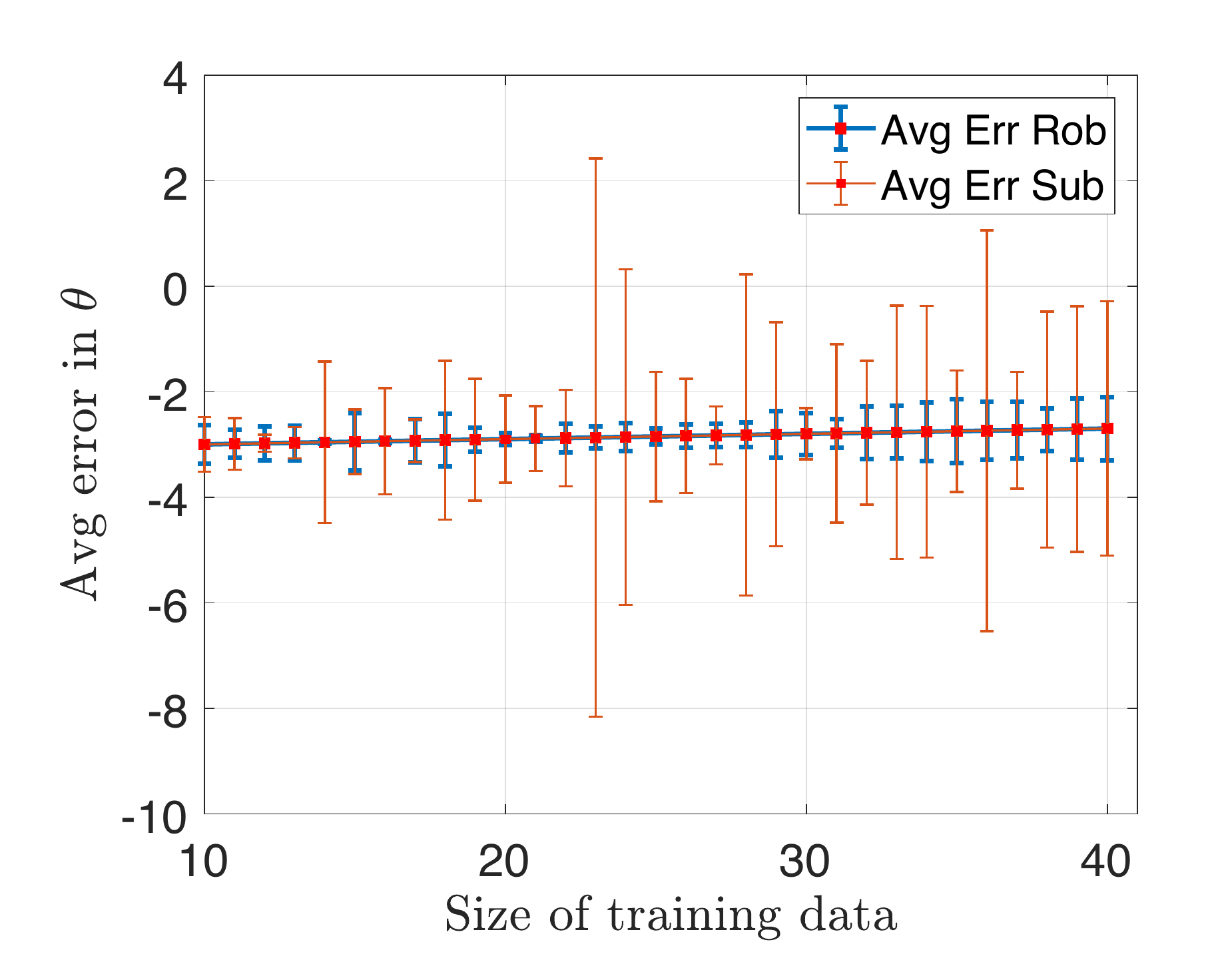}}
\caption{(a) Average error in prediction of $r$. (b) Average error in prediction of $\theta$.}\label{avg_err}
\end{figure}
In Fig. \ref{avg_err}(a) we plot the average error in prediction in $r$. The $x$-axis shows the number of time steps used for training and for each such training data set, the state $r$ was predicted for 10 future time steps and the average error in the prediction is plotted along the $y$-axis. It can be observed that for all the different sizes of training data used, average error for Robust DMD algorithm always shows smaller error compared to subspace DMD. Similarly, in Fig. \ref{avg_err}(b), we plot the average error is prediction of $\theta$ and observe the same.

\subsection{Stochastic Burger equation}
The third example we consider in this paper is  stochastic Burger's equation

\[\partial_t u(x,t)+u\partial_x u=k\partial_x^2u+\sigma_p e(x,t)\]

with $k=0.01$, $e(x,t)$ is uniform random distribution with support $[-1,1]$ and $\sigma_p=0.2$. In the simulation, we approximated the PDE solution using the Finite Difference method \cite{KUTLUAY1999251} with the initial condition $u(x,0)=sin(2\pi x)$ and Dirichet boundary condition $u(0,t)=u(1,t)=0$. Given the spatial and temporal ranges, $x\in[0,1],\;t\in[0,1]$, the discretizaion steps are chosen as $\Delta t=0.02$ and $\Delta x=1\times10^{-2}$.

The generated data is plotted in Fig. \ref{fig:PDE}(a). The robust DMD algorithm, subspace DMD algorithm and the regular DMD algorithm are applied to data set collected over 100 time steps. We assumed that both process and measurement noise and each has variance of 0.2. The eigenvalues obtained using regular DMD, subspace DMD and robust DMD approach are shown in Fig. \ref{fig:PDE})(b) and zoomed in plot shown in Fig. \ref{fig:PDE}(c). As in the previous examples, we find that one of the eigenvalues obtained using regular DMD and subspace DMD approach is unstable. Furthermore, the dominant eigenvalues show a closer match for robust DMD compared to other methods.


\begin{figure}[htp!]
\centering
\subfigure[\label{fig:PDE_data}]{\includegraphics[scale=.3]{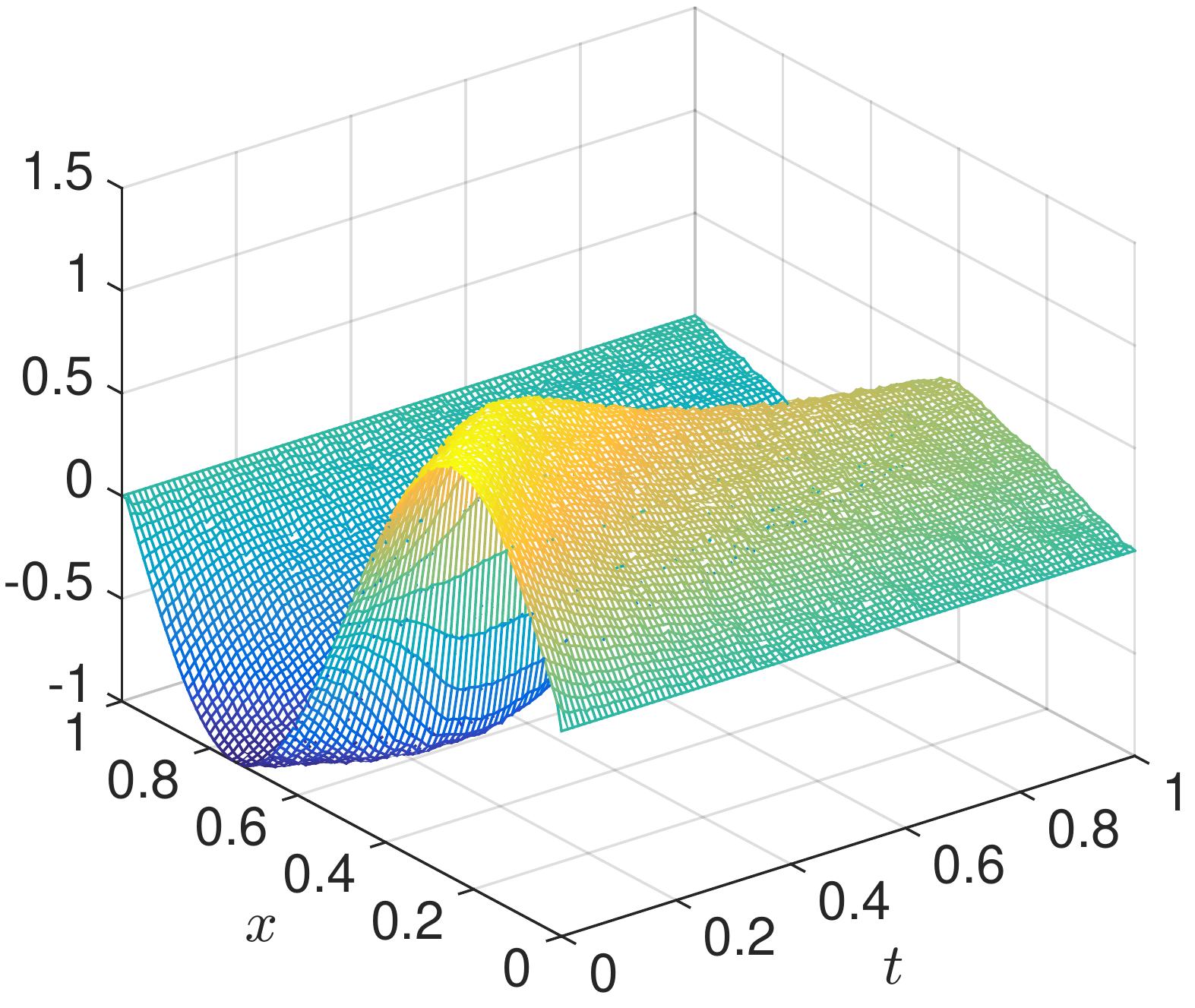}}
\subfigure[\label{fig:PDE_eigvalues}]{\includegraphics[scale=.275]{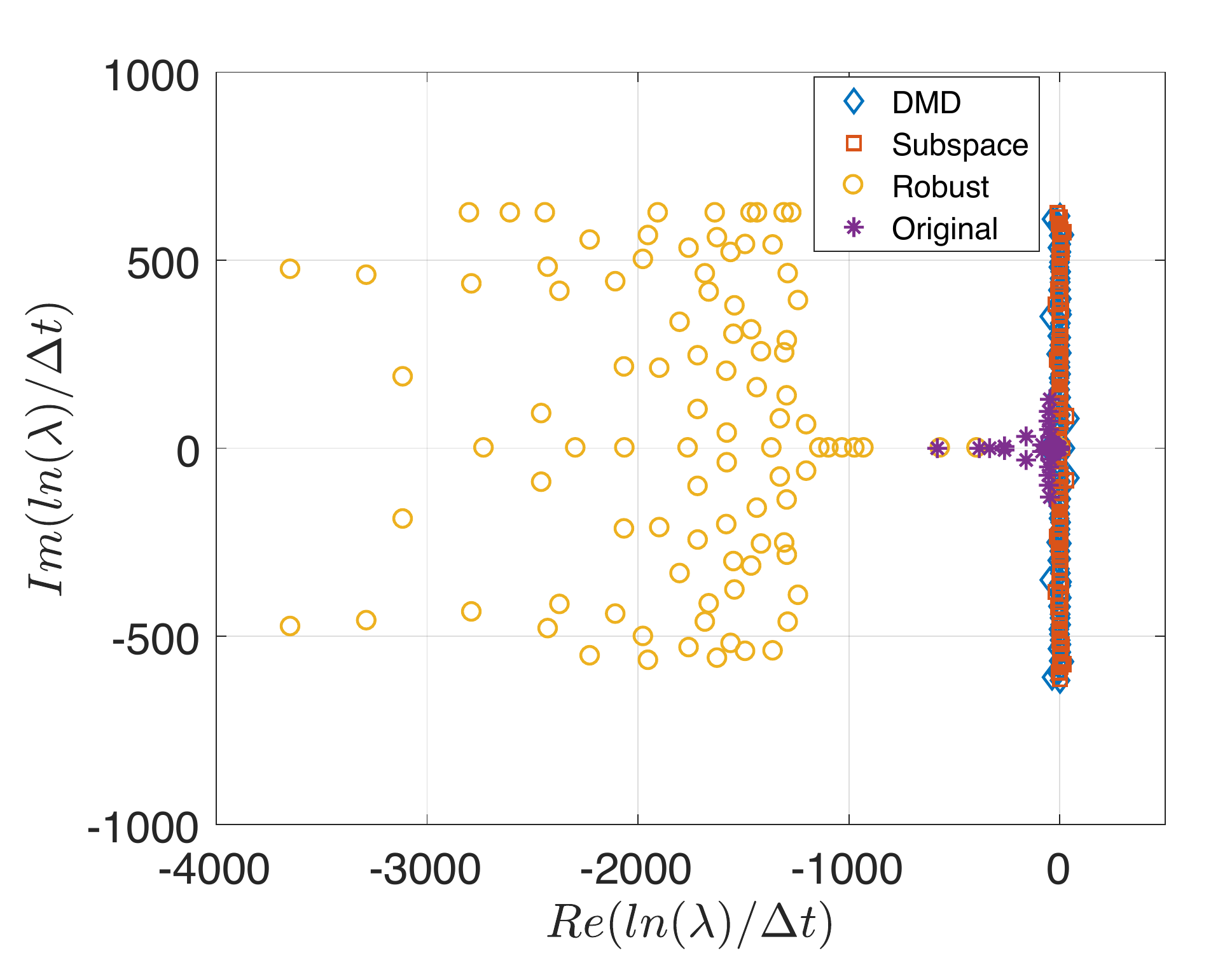}}
\subfigure[\label{fig:PDE_eigvalues2}]{\includegraphics[scale=.275]{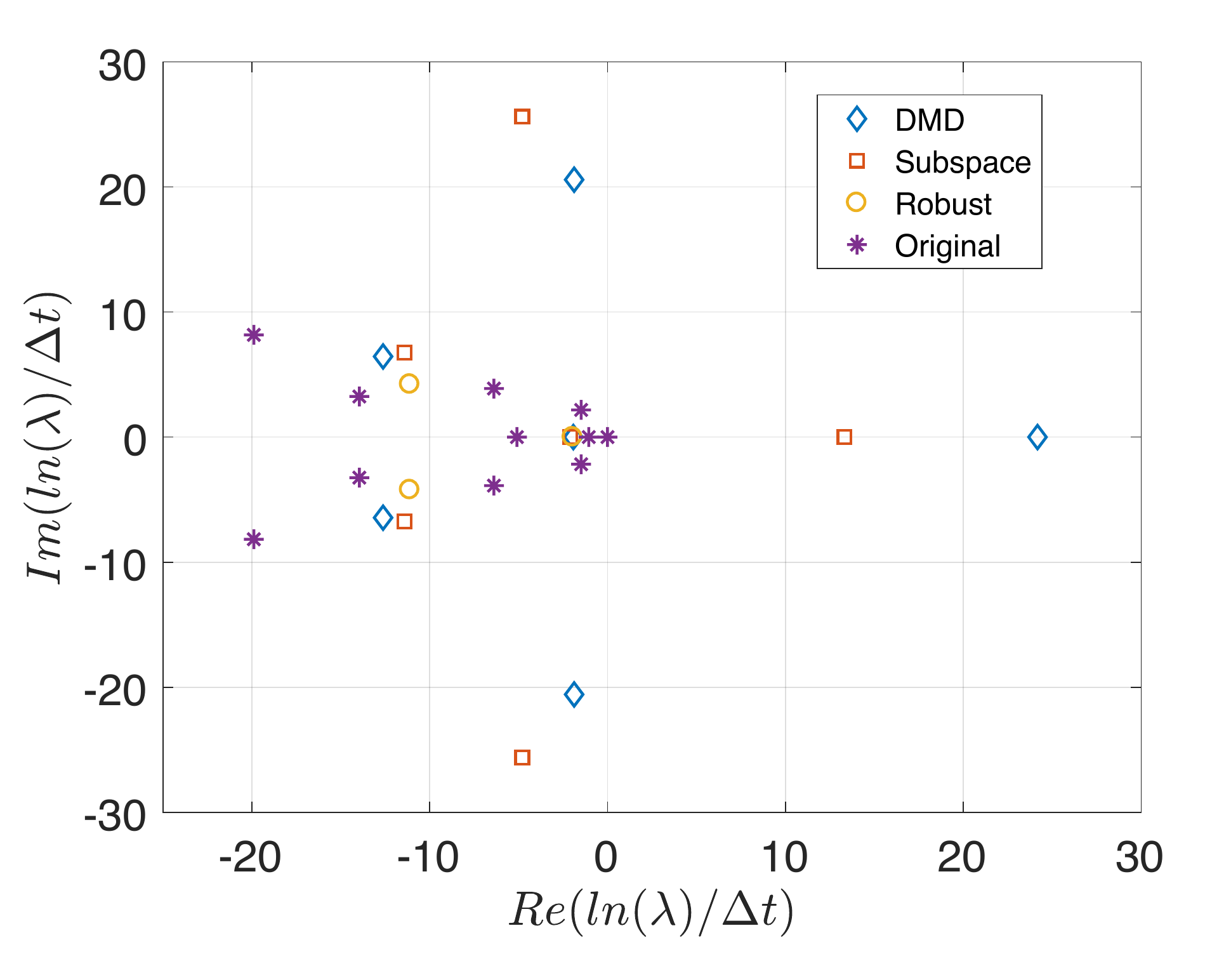}}
\caption{(a) Data generated by Stochastic Burger's equation with observation noise (b) Eigenvalues of the continuous time system.(c) Zoomed in region of dominant eigenvalues captured by the three methods.}\label{fig:PDE}

\end{figure}

Since the eigenvalues obtained with robust DMD algorithm shows a closer match, it is logical to believe that prediction using the robust DMD Koopman operator will be much more efficient compared to either regular DMD or subspace DMD.

\begin{figure}[htp!]
\centering
\subfigure[]{\includegraphics[scale=.25]{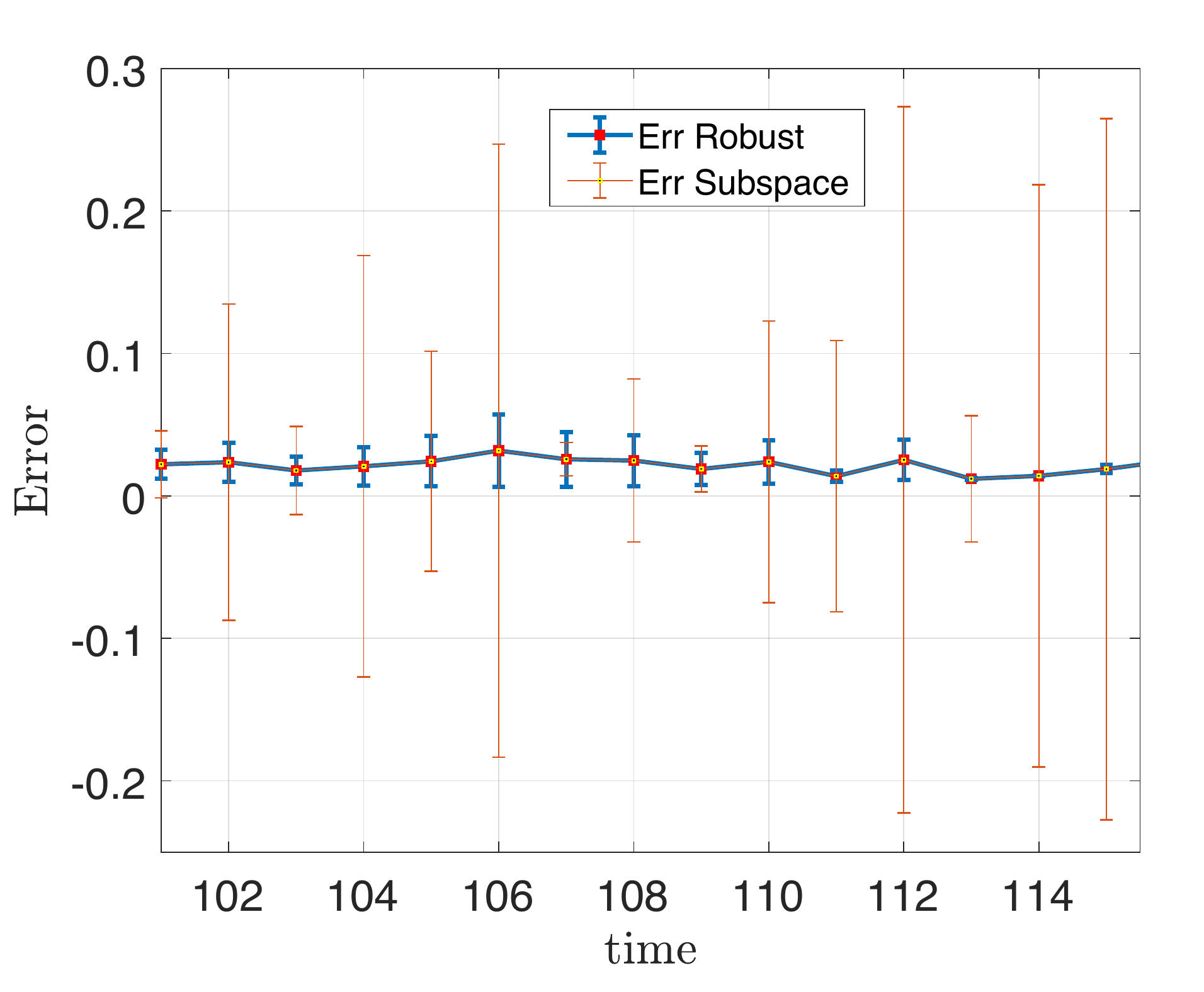}}
\subfigure[]{\includegraphics[scale=.25]{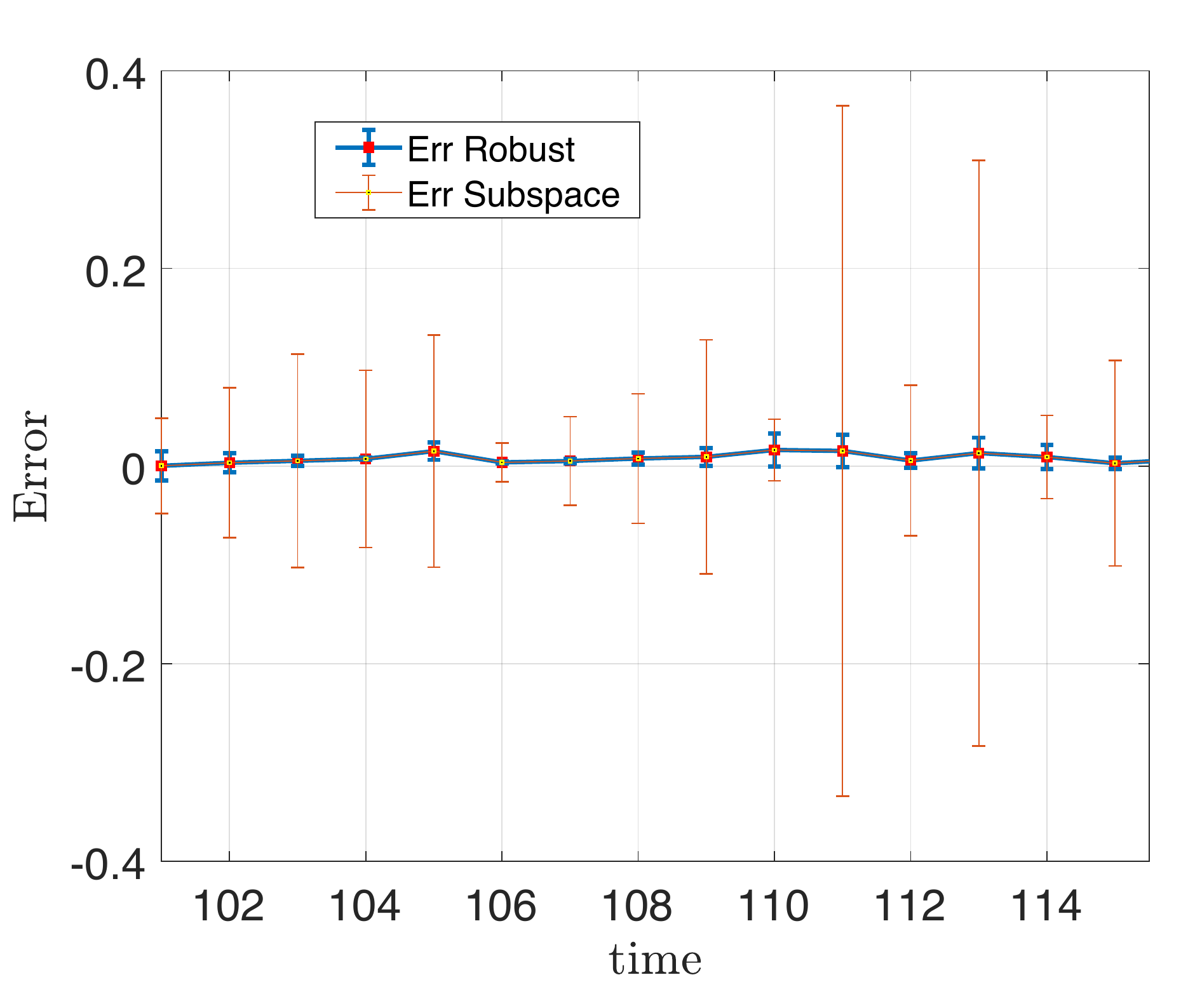}}
\subfigure[]{\includegraphics[scale=.25]{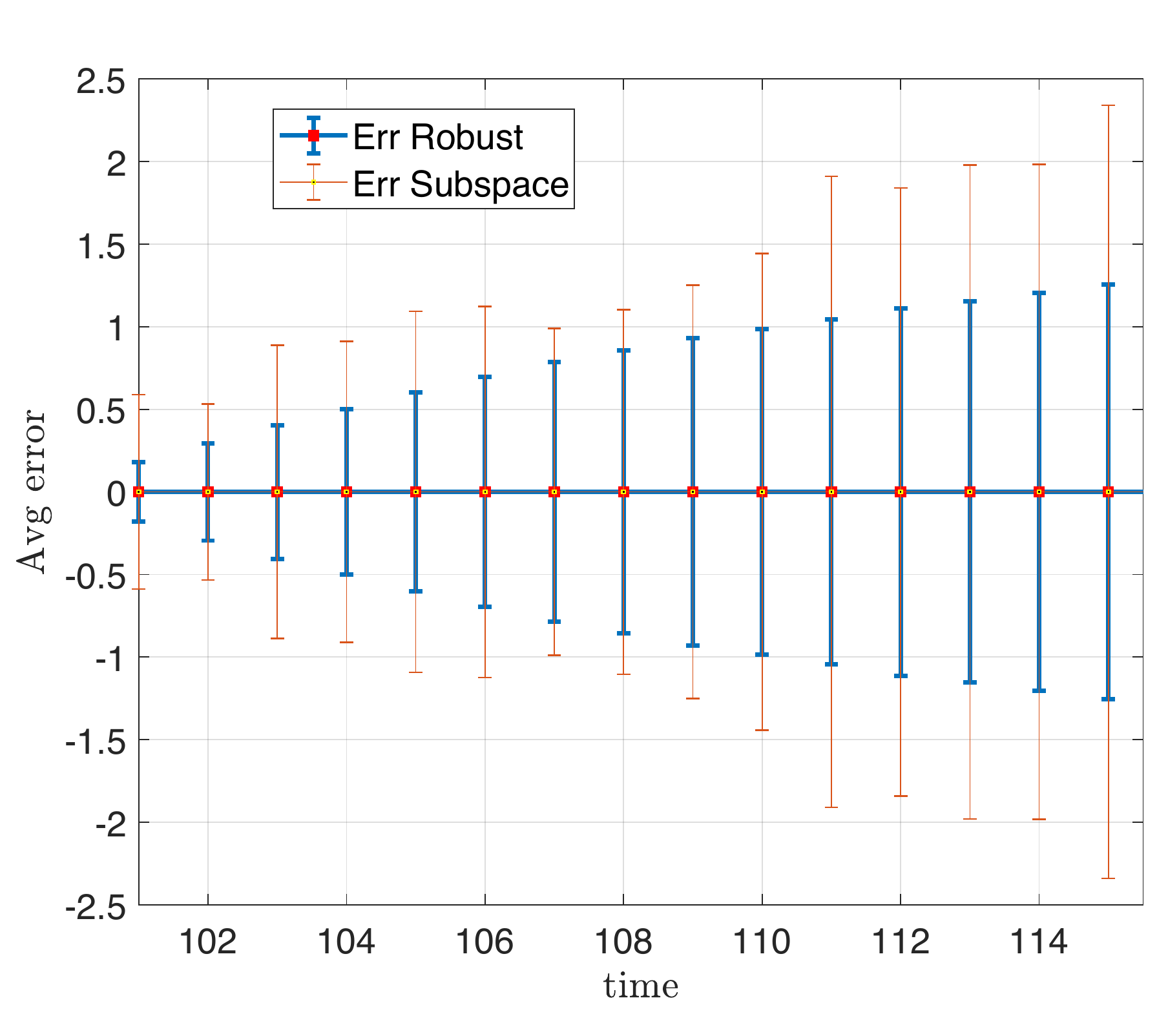}}
\caption{(a) Error in prediction of $z_2$. (b) Error in prediction of $z_{50}$. (c) Average error in prediction.}\label{pde_err}
\end{figure}

We used 100 time steps data as the training data and used the obtained operators to predict future 15 states. Since $\Delta x = 1\times 10^{-2}$, the discretized system has 100 states, namely, $\begin{pmatrix}
z_1 & z_2 & \cdots & z_{100}\end{pmatrix}$. In Fig. \ref{pde_err}(a) and (b) we plot the errors in prediction of $z_2$ and $z_{50}$ respectively, when using robust DMD and subspace DMD algorithms. It can be seen that the error in prediction with robust DMD is much smaller than the error obtained with subspace DMD. In Fig. \ref{pde_err}(c) we plot the average error in all the states at each prediction time step. We do not plot the error using regular DMD because since the Koopman operator obtained using regular DMD has an eigenvalue with large positive real part, it is highly unstable and the error diverges exponentially fast.

\section{Conclusions}\label{section_conclusion}
Robust optimization-based approach is proposed for the finite dimensional approximation of Koopman operator for dynamical system forced with process and measurement noise. The proposed approach leads to a better and stable approximation of Koopman operator compared to regular DMD-based approximation and subspace DMD algorithm. We showed that the proposed robust formulation for the approximation of Koopman operator allows us to balance the tradeoff between the quality of approximation and complexity of approximation. This allows us to make use of the robust formulation for the design of data-driven predictor dynamics for nonlinear systems. 

\bibliographystyle{IEEEtran}
\bibliography{ref,ref1,reference}

\end{document}